\documentclass[a4paper]{article}

\usepackage{amsmath}
\usepackage{amssymb}
\usepackage{mathrsfs}
\usepackage{amsthm}
\usepackage{tabularx}
\usepackage{enumerate} % roman numerals
\usepackage{floatpag}  % Float page style
\usepackage{longtable} % Tables over multiple pages
\usepackage{lscape} % landscape pages
\usepackage{tikz-cd}
\usepackage{tikz}
\usepackage{calc}
\usepackage{xfrac} 

\newtheorem{thm}{Theorem}[section]
\newtheorem{defn}[thm]{Definition}
\newtheorem{lem}[thm]{Lemma}
\newtheorem{prop}[thm]{Proposition}
\newtheorem{coro}[thm]{Corollary}

\newcommand{\myfrac}[3][0pt]{\genfrac{}{}{}{}{\raisebox{#1}{$#2$}}{\raisebox{-#1}{$#3$}}}

\usepackage[top=1in, bottom=1.25in, left=1.25in, right=1.25in]{geometry}

\floatpagestyle{empty}

\newcommand{\ml}[2]{\begin{tabular}{@{} >{$}#1<{$} @{}} #2 \end{tabular}}

\numberwithin{equation}{section}

\title{An infinite family of axial algebras}
\author{Madeleine Whybrow\thanks{Arbeitsgruppe Algebra, Geometrie und Computeralgebra, TU Kaiserslautern; mlw10@ic.ac.uk}}
\date{}

\begin{document}

\maketitle

\begin{abstract}
Axial algebras are non-associative algebras generated by semisimple idempotents, known as axes, that all obey a fusion rule. Axial algebras were introduced by Hall, Rehren and Shpectorov as a generalisation of the axioms of Majorana theory, which was in turn introduced by Ivanov as an axiomatisation of certain properties of the $2A$-axes of the Griess algebra $V_{\mathbb{M}}$. Axial algebras of Monster type are axial algebras whose axes obey the Monster, or Majorana, fusion rule. 

We construct an axial algebra of Monster type $M_{4A}$ over the polynomial ring $\mathbb{R}[t]$ that is generated by six axes whose Miyamoto involutions generate an elementary abelian group of order $4$. This construction automatically provides an infinite-parameter family $\{M(t)\}_{t \in \mathbb{R}}$ of axial algebras of Monster type each of which admit a unique Frobenius form. Moreover, we show that this form on $M(t)$ is positive definite if and only if $0 < t < \frac{1}{6}$ and also satisfies Norton's inequality if and only if $0 \leq t \leq \frac{1}{6}$. Finally, we show that the $4A$ axes of $M_{4A}$ obey a $C_2 \times C_2$-graded fusion rule giving a new infinite family of fusion rules.
\end{abstract}

\section{Introduction}

The Monster group $\mathbb{M}$ is the largest of the twenty-six sporadic groups and was first constructed by Griess \cite{Griess82} as the automorphism group of the Griess algebra $V_{\mathbb{M}}$, a 196,884-dimensional real commutative non-associative algebra. The Griess algebra is generated by idempotents known as $2A$-axes that are in bijection with the $2A$-involutions in the Monster group. 

Inspired by work of Sakuma \cite{Sakuma07} on vertex operator algebras, Ivanov \cite{Ivanov09} introduced Majorana theory as an axiomatisation of certain properties of the $2A$-axes of the Griess algebra. The objects at the centre of Majorana theory are known as Majorana algebras. The axioms of Majorana theory were further generalised by Hall, Rehren and Shpectorov \cite{HRS15} who introduced the definition of an axial algebra. 

An axial algebra is a commutative non-associative algebra generated by semisimple idempotents, known as axes, that all obey a fusion rule. This fusion rule governs the behaviour of the eigenvectors of the adjoint action of an axis on the whole of the algebra. Majorana algebras (including the Griess algebra) and Jordan algebras are both important examples of axial algebras.

Axial algebras that obey the Monster fusion rules are referred to here as axial algebras of Monster type. In this case, as in other examples of axial algebras, the fusion rule is $C_2$-graded which in particular means that from each axis we can construct an involution in the automorphism group of the algebra known as a Miyamoto involution. 

In Section \ref{sec:axial}, we give some basic definitions and results concerning axial algebras. In Sections \ref{sec:construction}, \ref{sec:4B} and \ref{sec:4A} we classify a certain class of axial algebras of Monster type that are generated by six axes whose Miyamoto involutions generate an elementary abelian group of order $4$. As part of this classification, we obtain a construction of a $12$-dimensional axial algebra $M_{4A}$ over the polynomial ring $\mathbb{R}[t]$. 

This is a significant new example in a number of ways. Firstly, axial algebras were introduced by Hall et al. \cite{HRS15} using an algebra geometric construction giving a variety of algebras. The variety of algebras constructed in \cite[Section 7]{HRS15} is $0$-dimensional. We also use this construction (as explained in Section \ref{sec:construction}). This means that we provide the first example of a variety of axial algebras that has positive dimension $1$. In particular, the indeterminate $t$ may take any value in $\mathbb{R}$, giving the first example of an infinite family $\{M(t)\}_{t \in \mathbb{R}}$ of axial algebras all with the same shape (as defined in Section \ref{sec:shape}). 

Futhermore, in Section \ref{sec:frobenius}, we show that the axial algebra $M(t)$ is a admits a positive definite Frobenius form if and only if $t \in (0, \frac{1}{6})$ and also obeys Norton's inequality if and only if $t \in [0, \frac{1}{6}]$. This shows that an axial algebra of Monster type that admits a Frobenius form does not necessarily obey Norton's inequality. However, the question of whether Norton's inequality, which is one of the Majorana axioms, is a consequence of the other axioms remains an open problem in Majorana theory. Our result suggests that this is indeed the case.  Moreover, this work shows that if $t \in (0, \frac{1}{6})$ then the algebra $M(t)$ is a Majorana algebra. This is the first example of an infinite family of Majorana algebras. 

This work also has implications for the computational methods that are being developed in the field. A computational approach to Majorana and axial algebras was begun by Seress \cite{Seress12} and has been continued by McInroy and Shpectorov \cite{MS18} and Pfeiffer and Whybrow \cite{PW18a}. The algorithms developed by these authors take the shape either of a Majorana algebra or an axial algebra of Monster type and attempt to construct the universal algebra with this shape over a field of characteristic $0$. 

This process is computationally expensive and many examples involving a large number of axes fail to complete, requiring too much time or memory. There are also a small but significant number of examples with a lower number of axes that have failed to complete. The family of axial algebras constructed in this paper comes from one such case. The existence of this infinite family of algebras shows that this example, as well as probably many others, lie outside the scope of the current computational methods. A potential improvement to the algorithms in \cite{MS18} and \cite{PW18a} would be to allow the construction over the field of rational functions in one or more indeterminates. 

Finally, in Section \ref{sec:4Afusion}, we study the $4A$ axes contained in the the algebra $M_{4A}$. We show that the eigenvectors of these axes in $M_{4A}$ obey a fusion rule. Moreover, if $t \notin \{1, 0, \frac{1}{2}, \frac{3}{8}\}$, then this fusion rule is $C_2 \times C_2$-graded. This gives an important new example of an infinite family of graded fusion rules. 

\section{Axial algebras}
\label{sec:axial}

In the following, we let $R$ be a commutative ring with identity and let $(V, \cdot)$ be a commutative (not necessarily associative) $R$-algebra. 

\subsection{Preliminaries}

\begin{defn}
A \emph{fusion rule} is a pair $(\mathcal{F}, *)$ such that $\mathcal{F} \subseteq R$ and $\mathcal{F} \times \mathcal{F} \rightarrow 2^{\mathcal{F}}$ is a symmetric map. 
\end{defn}

\begin{defn}
If $a \in V$ is an idempotent then for each $\lambda \in R$, we denote the eigenspace of the adjoint action of $a$ on $V$ with eigenvalue $\lambda$ by 
\[
V^{(a)}_{\lambda} = \{ v \in V \mid a \cdot v = \lambda v\}. 
\]
For each subset $\Lambda \subseteq R$, we define 
\[
V^{(a)}_{\Lambda} = \bigoplus_{\lambda \in \Lambda} V^{(a)}_{\lambda}
\]
with the convention that $V^{(a)}_{\emptyset} = \{0 \}$.
\end{defn}

\begin{defn}
If $a \in V$ is an idempotent and $(\mathcal{F}, *)$ is a fusion rule then $a$ is a $(\mathcal{F}, *)$-\emph{axis} if 
\begin{enumerate}[(i)]
\item  $\mathrm{ad}_a$ is semisimple and whenever $V^{(a)}_{\lambda}$ is non-zero, $\lambda \in \mathcal{F}$, i.e. the equation
\[
V = V_{\mathcal{F}}^{(a)} = \bigoplus_{\lambda \in \mathcal{F}} V^{(a)}_{\lambda} 
\]
holds;
\item for all $\lambda, \mu \in \mathcal{F}$ we have that
\[
V_{\lambda}^{(a)} \cdot V_{\mu}^{(a)} \subseteq V_{\lambda * \mu}^{(a)}.
\]
\end{enumerate}
\end{defn}

\begin{defn}
A $(\mathcal{F}, *)$-\emph{axial algebra} is a pair $(V, A)$ where $V$ is a commutative (not necessarily associative) $R$-algebra and $A \subseteq V$ is a generating set of $(\mathcal{F}, *)$-axes. 
\end{defn}

\begin{defn}
A $(\mathcal{F}, *)$-axis is \emph{primitive} if its $1$-eigenspace is one dimensional, i.e. $V_1^{(a)} = \langle a \rangle_R$. A $(\mathcal{F}, *)$-axial algebra $(V,A)$ is \emph{primitive} if $a$ is primitive for all $a \in A$. 
\end{defn}

\begin{defn}
Suppose that $T$ is an abelian group. A function $\mathrm{gr}: T \rightarrow P(\mathcal{F})$ is a $T$-\emph{grading} if the image of $\mathrm{gr}$, $\{\mathrm{gr}(t) \mid t \in T\}$, is a partition of $\mathcal{F}$, and for all $t, t' \in T$, if $\alpha \in \mathrm{gr}(t)$ and $\beta \in \mathrm{gr}(t')$ then $\alpha * \beta \subseteq \mathrm{gr}(tt')$.
\end{defn}

\begin{defn}
A \emph{Frobenius form} on a commutative algebra $V$ is a (non-zero) bilinear form $\langle \, , \, \rangle$ such that for all $u, v, w \in V$
\[
\langle u, v \cdot w \rangle  = \langle u \cdot v, w \rangle. 
\]
We say that an algebra is \emph{Frobenius} if it admits such a form.
\end{defn}

Henceforth, as is customary, we require that a Frobenius form $\langle \, , \, \rangle$ on an axial algebra $V$ satisfies $\langle a, a \rangle = 1$ for all axes $a \in V$.

\begin{prop}
\label{prop:orthogonality}
Let $(V, A)$ be a $(\mathcal{F}, *)$-axial algebra over a ring $R$ that admits a Frobenius form $\langle \, , \, \rangle$. Suppose that $R$ is an integral domain and suppose that $\lambda^{-1} \in R$ for all $\lambda \in \mathcal{F}$ such that $\lambda \neq 0$. Then the eigenspace decomposition $V = \bigoplus_{\lambda \in \mathcal{F}} V_{\lambda}^{(a)}$ is orthogonal with respect to $\langle \, , \, \rangle$ (i.e. $\langle u,v \rangle = 0$ for all $u \in V_{\mu}^{(a)}$, $v \in V_{\nu}^{(a)}$ with $\mu, \nu \in \mathcal{F}$ and $\mu \neq \nu$).
\end{prop}

\begin{proof}
Let $u \in V_{\mu}^{(a)}$, $v \in V_{\nu}^{(a)}$ with $\mu \neq \nu$ as in the hypothesis. Suppose first that $\mu = 0$ and $\nu \neq 0$. Then
\[
\langle u, v \rangle = \frac{1}{\nu}\langle u, a \cdot v \rangle = \frac{1}{\nu}\langle a \cdot u, v \rangle = 0 
\]
If both $\mu$ and $\nu$ are non-zero then 
\[
\frac{1}{\mu}\langle a \cdot u, v \rangle = \langle u,v \rangle = \frac{1}{\nu}\langle u, a \cdot v \rangle = \frac{1}{\nu}\langle a \cdot u, v \rangle
\]
and so, as $\mu \neq \nu$, $\langle a \cdot u, v \rangle = \langle u,v \rangle = 0$. 
\end{proof}

\begin{lem}
\label{lem:projection}
Let $(V,A)$ be as in Proposition \ref{prop:orthogonality}, let $a \in A$ be a primitive axis and let $v \in V$. Then the projection of $v$ onto $\langle a \rangle$ is equal to $\langle a, v \rangle$.
\end{lem}

\begin{proof}
As $V$ is a axial algebra, we can write
\[
v = \bigoplus_{\lambda \in \mathcal{F}} v_\lambda
\]
where $v_\lambda \in V_{\lambda}^{(a)}$. As $a$ is primitive, $a_1 = \phi a$ for some $\phi \in R$. From Proposition \ref{prop:orthogonality}, if $\lambda \neq 1$ then $\langle a, v_\lambda \rangle = 0$ and so 
\[
\langle a, v \rangle = \langle a, \phi a \rangle = \phi
\] 
as required.
\end{proof} 

\subsection{Axial algebras of Monster type}

Henceforth, we assume that the ring $R$ is an integral domain that contains a subfield $\mathbf{k}$ of characteristic $0$ such that $\frac{1}{4}, \frac{1}{32} \in \mathbf{k}$.
\begin{defn}
The \emph{Monster} or \emph{Majorana} fusion rule $(\mathcal{M}, *)$ is given by the following table where the $(\lambda, \mu)$-entry gives the entries of the set $\lambda * \mu$. As it customary, we omit the brackets from the sets in this table.
\begin{center}
\def\arraystretch{1.5}
\begin{tabular}{>{$}c<{$}|>{$}c<{$}>{$}c<{$}>{$}c<{$}>{$}c<{$}}
  & 1 & 0 & \frac{1}{4} & \frac{1}{32} \\ \hline
1 & 1 & \emptyset  & \frac{1}{4} & \frac{1}{32} \\
0 & \emptyset  & 0 & \frac{1}{4} & \frac{1}{32} \\
\frac{1}{4} & \frac{1}{4} & \frac{1}{4} & 1,0 & \frac{1}{32} \\
\frac{1}{32} & \frac{1}{32} & \frac{1}{32} & \frac{1}{32} & 1,0,\frac{1}{4} 
\end{tabular}
\end{center}
A $(\mathcal{M}, *)$-axial algebra is called an axial algebra \emph{of Monster type}.
\end{defn}

The $2A$-axes of the Griess algebra $V_{\mathbb{M}}$ are known to obey the Monster fusion rule. As the $2A$-axes generate $V_{\mathbb{M}}$, it is an example of an axial algebra of Monster type.

The Monster fusion rule admits a $C_2$-grading $\mathrm{gr}: \langle a \mid a^2 = 1 \rangle \rightarrow P(\mathcal{M})$ such that $\mathrm{gr}(1) = \{1, 0, \frac{1}{2}\}$ and $\mathrm{gr}(a) = \{\frac{1}{32} \}$. From this grading we can construct certain involutions in the automorphism group of the algebra, as described below. 

\begin{defn}
Suppose that $(V, A)$ is an axial algebra of Monster type. Then for each $a \in A$, we can define a map $\tau(a) \in \mathrm{GL}(V)$ such that 
\begin{align*}
\tau(a): v \mapsto
\begin{cases}
v &\textrm{ if } v \in V_1^{(a)} \oplus V_0^{(a)} \oplus V_{\frac{1}{4}}^{(a)} \\
- v &\textrm{ if } v \in V_{\frac{1}{32}}^{(a)}. 
\end{cases}
\end{align*}
Then $\tau(a)$ is called a Miyamoto involution.
\end{defn}

In the case of the Griess algebra, the Miyamoto involutions $\tau(a)$, where $a$ is a $2A$-axis, form the $2A$ conjugacy class of involutions of the Monster group.

\begin{prop}
\label{prop:algproduct}
Suppose that $(V, A)$ is an axial algebra of Monster type and that $a \in A$. Then the Miyamoto involution $\tau(a)$ preserves the algebra product on $V$, i.e. $u^{\tau(a)} \cdot v ^{\tau(a)} = (u \cdot v)^{\tau(a)}$ for all $u, v \in V$. 
\end{prop}

\begin{proof}
This follows directly from the definition of $\tau(a)$ and the Monster fusion rule. 
\end{proof}

\begin{prop}
\label{prop:form}
Let $(V, A)$ be an axial algebra of Monster type over a ring $R$ that admits a Frobenius form $\langle \, , \, \rangle$. Suppose that $R$ is an integral domain and suppose that $\lambda^{-1} \in R$ for all $\lambda \in \mathcal{F}$ such that $\lambda \neq 0$. If $a \in A$ then $\tau(a)$ also preserves the Frobenius form, i.e. $\langle u^{\tau(a)}, v ^{\tau(a)} \rangle = \langle u,v  \rangle$ for all $u, v \in V$. 
\end{prop}

\begin{proof}
Suppose that $u, v \in V$. We can write $u = u_1 + u_0 + u_{\frac{1}{4}} + u_{\frac{1}{32}}$ and $v = v_1 + v_0 + v_{\frac{1}{4}} + v_{\frac{1}{32}}$ where $u_{\lambda}, v_{\lambda} \in V_{\lambda}^{(a)}$ for $\lambda \in \{1, 0, \frac{1}{4}, \frac{1}{32}\}$. Then using Proposition \ref{prop:orthogonality}, we obtain
\begin{align*}
\langle u^{\tau(a)}, v^{\tau(a)} \rangle &= \langle u_1 + u_0 + u_{\frac{1}{4}} - u_{\frac{1}{32}}, v_1 + v_0 + v_{\frac{1}{4}} - v_{\frac{1}{32}} \rangle \\
    &= \langle u_1, v_1 \rangle + \langle u_0, v_0 \rangle + \langle u_{\frac{1}{4}}, v_{\frac{1}{4}} \rangle + \langle u_{\frac{1}{32}}, v_{\frac{1}{32}} \rangle
\end{align*}
and 
\begin{align*}
\langle u, v \rangle &= \langle u_1 + u_0 + u_{\frac{1}{4}} + u_{\frac{1}{32}}, v_1 + v_0 + v_{\frac{1}{4}} + v_{\frac{1}{32}} \rangle \\
    &= \langle u_1, v_1 \rangle + \langle u_0, v_0 \rangle + \langle u_{\frac{1}{4}}, v_{\frac{1}{4}} \rangle + \langle u_{\frac{1}{32}}, v_{\frac{1}{32}} \rangle.
\end{align*}
Thus $\langle u^{\tau(a)}, v ^{\tau(a)} \rangle = \langle u,v \rangle$ as required.
\end{proof}

\subsection{Dihedral axial algebras of Monster type}

\begin{defn}
Suppose that $(V,A)$ is an axial algebra and that $a, b \in A$ and $ a \neq b$. Then the subalgebra $U$ of $V$ generated by $a$ and $b$ is called a \emph{dihedral algebra}.
\end{defn}

The dihedral algebras of the Griess algebra (i.e. those generated by two distinct $2A$-axes) were studied by Conway and Norton \cite{Conway84} and were shown to have eight possible isomorphism types. A milestone in Majorana theory was reached when Ivanov, Pasechnik, Seress and Shpectorov \cite{IPSS10} proved that a Majorana algebra generated by two Majorana axes is isomorphic to one of the dihedral subalgebras of the Griess algebra. This reproved Conway and Norton's classification of the dihedral subalgebras of the Griess algebra, providing a foundation for the theory of Majorana algebras. 

In 2015, Hall, Rehren and Shpectorov reproved this result for primitive axial algebras of Monster type over a field of characteristic $0$ that admit a Frobenius form. 

\begin{thm}[{\cite[Theorem 1.2]{HRS15}}]
\label{thm:IPSS10}
Let $\mathbf{k}$ be a field such that $\mathrm{char}(\mathbf{k}) = 0$. Suppose that $(V,A)$ is a primitive axial algebra of Monster type over $\mathbf{k}$ that admits a Frobenius form. Let $a_0, a_1 \in A$ such that $a_0 \neq a_1$ and let $U$ be the subalgebra of $V$ generated by $a_0$ and $a_1$. Finally, let $\rho = \tau(a_0)\tau(a_1)$. Then
\begin{enumerate}[(i)]
\item the subalgebra $U = \langle \langle a_0, a_1 \rangle \rangle$ is isomorphic to a dihedral algebra of type $NX$, the structure of which is given in Table \ref{tab:IPSS10};
\item for $i \in \mathbb{Z}$ and $\epsilon \in \{0,1\}$, the image of $a_{\epsilon}$ under the $i$-th power of $\rho$, which we denote $a_{2i+\epsilon}$, is a $(\mathcal{M}, *)$-axis and $\tau(a_{2i + \epsilon}) = \rho^{-i}\tau_{\epsilon}\rho^i$.
\end{enumerate}
\end{thm}

We note that Table \ref{tab:IPSS10} does not show all values of the algebra product and Frobenius form on the dihedral algebras. Those which are omitted can be recovered from the action of the group $D$ and the symmetry between $a_0$ and $a_1$. The following useful lemma also follows from these symmetries. 

\begin{prop}
\label{prop:axes}
Let $\mathbf{k}$ be a field such that $\mathrm{char}(\mathbf{k}) = 0$. Suppose that $(V,A)$ is a primitive axial algebra of Monster type over $\mathbf{k}$ that admits a Frobenius form. Let $a_0, a_1 \in A$ such that $a_0 \neq a_1$ and let $U$ and $a_{2i + \epsilon}$ for $\epsilon \in \{0, 1\}$ and $i \in \mathbb{Z} \backslash \{0\}$ be as in Theorem \ref{thm:IPSS10}.

If $U$ is of type $3A$ or $4A$ then $U$ contains the additional basis vector $u_{\rho(a_0, a_1)}$ or $v_{\rho(a_0, a_1)}$. These vectors obey the following equalities
\[
u_{\rho(a_0, a_1)} = u_{\rho(a_0, a_{-1})} \mbox{ and } v_{\rho(a_0, a_1)} = v_{\rho(a_0, a_{-1})}.
\]
Similarly, if $U$ is of type $5A$ then $U$ contains an additional basis vector $w_{\rho(a_0, a_1)}$ that obeys the following equalities
\begin{equation}
\label{eq:5Aaxes}
w_{\rho(a_0, a_1)} = - w_{\rho(a_0, a_2)} = - w_{\rho(a_0, a_{-2})} = w_{\rho(a_0, a_{-1})}.
\end{equation}

Moreover, if $U := \langle \langle a_0, a_1 \rangle \rangle$ is a dihedral algebra of type $3A$, $4A$ or $5A$ then $u_{\rho(a_0, a_1)} = u_{\rho(a_1, a_0)}$, $v_{\rho(a_0, a_1)} = v_{\rho(a_1, a_0)}$ or $w_{\rho(a_0, a_1)} = w_{\rho(a_1, a_0)}$ respectively.
\end{prop}

\begin{defn}
Using notation as in Proposition \ref{prop:axes} above, we refer to the vectors $u_{\rho(a_0, a_1)}$, $v_{\rho(a_0, a_1)}$ and $w_{\rho(a_0, a_1)}$ as $3A$, $4A$ and $5A$-\emph{axes} respectively.
\end{defn}

\begin{table}
\begin{center}
\vspace{0.35cm}
\noindent
\begin{tabular}{|c|c|c|}
\hline
&&\\
 Type & Basis & Products and angles \\
&&\\
\hline
&&\\
&& $a_0 \cdot a_1=\frac{1}{2^3}(a_0+a_1-a_{\rho(a_0, a_1)}),~a_0 \cdot a_{\rho(a_0, a_1)}=\frac{1}{2^3}(a_0+a_{\rho(a_0, a_1)}-a_1)$ \\
2A & $a_0,a_1,a_{\rho(a_0, a_1)}$ & $a_{\rho(a_0, a_1)} \cdot a_{\rho(a_0, a_1)} = a_{\rho(a_0, a_1)}$ \\ 
&&$(a_0,a_1)=(a_0,a_{\rho(a_0, a_1)})=(a_{\rho(a_0, a_1)},a_{\rho(a_0, a_1)})=\frac{1}{2^3}$\\
&& \\ 
2B & $a_0,a_1$ &$a_0 \cdot a_1=0$,~$(a_0,a_1)=0$ \\
&&\\
&  &$a_0 \cdot a_1=\frac{1}{2^5}(2a_0+2a_1+a_{-1})-\frac{3^3 \cdot 5}{2^{11}}u_{\rho(a_0, a_1)}$\\
3A& $a_{-1},a_0,a_1,$ & $a_0 \cdot u_{\rho(a_0, a_1)}=\frac{1}{3^2}(2a_0-a_1-a_{-1})+\frac{5}{2^5}u_{\rho(a_0, a_1)}$~~~~\\
&$u_{\rho(a_0, a_1)}$& $u_{\rho(a_0, a_1)} \cdot u_{\rho(a_0, a_1)}=u_{\rho(a_0, a_1)}$\\
&& $(a_0,a_1)=\frac{13}{2^8}$,~$(a_0,u_{\rho(a_0, a_1)})=\frac{1}{2^2}$,~$(u_{\rho(a_0, a_1)},u_{\rho(a_0, a_1)})=\frac{2^3}{5}$
\\
&&\\
3C & $a_{-1},a_0,a_1$ & $a_0 \cdot a_1=\frac{1}{2^6}(a_0+a_1-a_{-1}),~(a_0,a_1)=\frac{1}{2^6}$ \\
&&\\ 
&  & ~$a_0 \cdot a_1=\frac{1}{2^6}(3a_0+3a_1+a_2+a_{-1}-3v_{\rho(a_0, a_1)})$\\
4A & $a_{-1},a_0,a_1,$ & $a_0 \cdot v_{\rho(a_0, a_1)}=\frac{1}{2^4}(5a_0-2a_1-a_2-2a_{-1}+3v_{\rho(a_0, a_1)})$\\
&$a_2,v_{\rho(a_0, a_1)}$&~$v_{\rho(a_0, a_1)} \cdot v_{\rho(a_0, a_1)}=v_{\rho(a_0, a_1)}$, ~$a_0 \cdot a_2=0$ \\
& & $(a_0,a_1)=\frac{1}{2^5},~(a_0,a_2)=0,~(a_0,v_{\rho(a_0, a_1)})=\frac{3}{2^3},~(v_{\rho(a_0, a_1)},v_{\rho(t_0,t_2)})=2$\\
&&\\
4B & $a_{-1},a_0,a_1,$ & $a_0 \cdot a_1=\frac{1}{2^6}(a_0+a_1-a_{-1}-a_2+a_{\rho(t_0,t_2)})$
\\
& $a_2,a_{\rho(t_0,t_2)}$ & $a_0 \cdot a_2=\frac{1}{2^3}(a_0+a_2-a_{\rho(t_0,t_2)})$\\
&& $(a_0,a_1)=\frac{1}{2^6},~(a_0,a_2)=(a_0,a_{\rho(a_0, a_1)})=\frac{1}{2^3}$ \\
&&\\
&& $a_0 \cdot a_1=\frac{1}{2^7}(3a_0+3a_1-a_2-a_{-1}-a_{-2})+w_{\rho(a_0, a_1)}$
\\
 5A & $a_{-2},a_{-1},a_0,$ & $a_0 \cdot a_2=\frac{1}{2^7}(3a_0+3a_2-a_1-a_{-1}-a_{-2})-w_{\rho(a_0, a_1)}$
\\
& $a_1,a_2,w_{\rho(a_0, a_1)}$ & $a_0 \cdot w_{\rho(a_0, a_1)}=\frac{7}{2^{12}}(a_{1}+a_{-1}-a_2-a_{-2})+\frac{7}{2^5}w_{\rho(a_0, a_1)}$\\
& & $w_{\rho(a_0, a_1)} \cdot w_{\rho(a_0, a_1)}=\frac{5^2 \cdot 7}{2^{19}}(a_{-2}+a_{-1}+a_0+a_1+a_2)$\\
&&$(a_0,a_1)=\frac{3}{2^7},~(a_0,w_{\rho(a_0, a_1)})=0$, $(w_{\rho(a_0, a_1)},w_{\rho(a_0, a_1)})=\frac{5^3 \cdot 7}{2^{19}}$\\
&& \\
& & $a_0 \cdot a_1=\frac{1}{2^6}(a_0+a_1-a_{-2}-a_{-1}-a_2-a_3+a_{\rho(t_0,t_3)})+\frac{3^2 \cdot 5}{2^{11}}u_{\rho(t_0,t_2)}$\\
6A& $a_{-2},a_{-1},a_0,$ &$a_0 \cdot a_2=\frac{1}{2^5}(2a_0+2a_2+a_{-2})-\frac{3^3 \cdot 5}{2^{11}}u_{\rho(t_0,t_2)}$  \\ 
&$a_1,a_2,a_3$  &$a_0 \cdot u_{\rho(t_0,t_2)}=\frac{1}{3^2}(2a_0-a_2-a_{-2})+\frac{5}{2^5}u_{\rho(t_0,t_2)}$  \\
&$a_{\rho(t_0,t_3)},u_{\rho(t_0,t_2)}$ & $a_0 \cdot a_3=\frac{1}{2^3}(a_0+a_3-a_{\rho(t_0,t_3)})$, $a_{\rho(t_0,t_3)} \cdot u_{\rho(t_0,t_2)}=0$\\
&&$(a_{\rho(t_0,t_3)},u_{\rho(t_0,t_2)})=0$, $(a_0,a_1)=\frac{5}{2^8}$, $(a_0,a_2)=\frac{13}{2^8}$, $(a_0,a_3)=\frac{1}{2^3}$\\
&&\\
\hline
\end{tabular}
\caption{The dihedral axial algebras of Monster type}
\label{tab:IPSS10}  
\end{center}
\end{table}

Using the algebra values in Table \ref{tab:IPSS10}, we can also calculate the eigenspace decompositions of the dihedral axial algebras of Monster type. Bases of the $0$, $\frac{1}{4}$ and $\frac{1}{32}$-eigenspaces of the axis $a_0$ for each dihedral algebra are given in Table \ref{tab:dihedralevecs}. In each case, the $1$-eigenspace is, by definition, one dimensional and so is omitted from this table.

\begin{table}
\begin{center}
\vspace{0.35cm}
\noindent
\begin{tabular}{|>{$}c<{$}|>{$}c<{$}|>{$}c<{$}|>{$}c<{$}|}
\hline
&&&\\
 \mbox{Type} & 0 & \frac{1}{4} & \frac{1}{32}  \\
&&&\\
\hline
&&&\\
2A & a_1 + a_{\rho(a_0, a_1)} - \frac{1}{2^2} & a_1 - a_{\rho(a_0, a_1)} & \\
&&&\\
2B & a_1 & & \\
&&&\\
3A  & u_{\rho(a_0, a_1)} - \frac{2 \cdot 5}{3^3}a_0 + \frac{2^5}{3^3}(a_1 + a_{-1}) 
    & \ml{c}{u_{\rho(a_0, a_1)} - \frac{2^3}{3^2 \cdot 5}a_0\\ 
        - \frac{2^5}{3^2 \cdot 5}(a_1 + a_{-1}) }
    & a_1 - a_{-1} \\
&&&\\
3C  & a_1 + a_{-1} - \frac{1}{2^5}a_0 
    & 
    & a_1 - a_{-1} \\ 
&&&\\
4A  & v_{\rho(a_0, a_1)} - \frac{1}{2}a_0 + 2(a_1 + a_{-1}), a_2 
    & \ml{c}{v_{\rho(a_0, a_1)} - \frac{1}{3}a_0 \\ 
        - \frac{2}{3}(a_1 + a_{-1}) - \frac{1}{3}a_2} 
    & a_1 - a_{-1} \\ 
&&&\\
4B  & \ml{c}{a_1 + a_{-1} - \frac{1}{2^5}a_0 - \frac{1}{2^3}(a_{\rho(t_0,t_2)}- a_2), \\ 
        a_2 + a_{\rho(t_0,t_2)} - \frac{1}{2^2}a_0} 
    & a_2 - a_{\rho(t_0,t_2)} 
    & a_1 - a_{-1} \\ 
&&&\\
5A  & \ml{c}{w_{\rho(a_0, a_1)} + \frac{3}{2^9}a_0 - \frac{3 \cdot 5}{2^7}(a_1 + a_{-1}) - \frac{1}{2^7}(a_2 - a_{-2}), \\ 
        w_{\rho(a_0, a_1)} - \frac{3}{2^9}a_0 + \frac{1}{2^7}(a_1 + a_{-1}) + \frac{3 \cdot 5}{2^7}(a_2 + a_{-2})} 
    & \ml{c}{w_{\rho(a_0, a_1)} + \frac{1}{2^7}(a_1 + a_{-1}) \\
        - \frac{1}{2^7}(a_2 + a_{-2})} 
    & \ml{c}{a_1 - a_{-1}, \\ a_2 - a_{-2}} \\
&&&\\
6A  & \ml{c}{u_{\rho(t_0,t_2)} + \frac{2}{3^2 \cdot 5}a_0 - \frac{2^8}{3^2 \cdot 5}(a_1 - a_{-1}) \\ 
        - \frac{2^5}{3^2 \cdot 5}(a_2 + a_{-2} + a_3 - a_{\rho(t_0,t_3)}), \\ 
        a_3 + a_{\rho(t_0,t_3)} - \frac{1}{2^2}a_0, \\ 
        u_{\rho(t_0,t_2)} - \frac{2 \cdot 5}{3^3}a_0 + \frac{2^5}{3^3}(a_2 + a_{-2})} 
    & \ml{c}{u_{\rho(t_0,t_2)} - \frac{2^3}{3^2 \cdot 5}a_0  \\ 
        - \frac{2^5}{3^2 \cdot 5}(a_2 + a_{-2} + a_3)\\
        + \frac{2^5}{3^2 \cdot 5} a_{\rho(t_0,t_3)} , \\ 
        a_3 - a_{\rho(t_0,t_3)}} 
    & \ml{c}{a_1 - a_{-1}, \\ a_2 - a_{-2}} \\
&&& \\
\hline
\end{tabular}
\caption{The eigenspace decomposition of the dihedral axial algebras of Monster type}
\label{tab:dihedralevecs}  
\end{center}
\end{table}

The following results follow directly from the values in Table \ref{tab:IPSS10}.

\begin{lem}
\label{lem:dihedral}
Let $U$ be a dihedral axial algebra of Monster type (as in Table \ref{tab:IPSS10}) that is generated by axes $a_0$ and $a_1$. Let $D := \langle \tau(a_0), \tau(a_1) \rangle$. Then $|a_0^D \cup a_1^D| = N$.
\end{lem}

\begin{lem}
\label{lem:inclusions}
Let $U$ be a dihedral axial algebra of Monster type (as in Table \ref{tab:IPSS10}) that is generated by axes $a_0$ and $a_1$. Then $U$ contains no proper, non-trivial subalgebras, with the exception of the following cases.
\begin{enumerate}[(i)]
\item If $U$ is of type $4A$ or $4B$ then the subalgebras $\langle \langle a_0, a_2 \rangle \rangle$ and $\langle \langle a_1, a_{-1} \rangle \rangle$ are of type $2B$ or $2A$ respectively.
\item If $U$ is of type $6A$ then the subalgebras $\langle \langle a_0, a_3 \rangle \rangle$ and $\langle \langle a_1, a_{-2} \rangle \rangle$ are of type $3A$ and the subalgebras $\langle \langle a_0, a_2 \rangle \rangle$, $\langle \langle a_1, a_{-1} \rangle \rangle$, $\langle \langle a_0, a_{-2} \rangle \rangle$ and $\langle \langle a_1, a_3 \rangle \rangle$ are of type $2A$. 
\end{enumerate} 
\end{lem}

Informally, this means that we have the following \emph{inclusions} of algebras:
\[
2A \hookrightarrow 4B, \qquad 2B \hookrightarrow 4A, \qquad 2A \hookrightarrow 6A, \qquad \textrm{and} \qquad 3A \hookrightarrow 6A
\]
and that these are the only possible inclusions of non-trivial algebras. 

\section{Construction}
\label{sec:construction}

We will classify Frobenius $(\mathcal{M}, *)$-axial algebras that obey property $\mathcal{P}$ as defined below. In doing so, we will construct an infinite family of Frobenius $(\mathcal{M}, *)$-axial algebras over the field $\mathbb{R}$. 

\begin{defn}
We say that a Frobenius $(\mathcal{M}, *)$-axial algebra $V$ obeys \emph{property $\mathcal{P}$} if it is generated by a set of axes $A := \{a_1, a_{-1}, a_2, a_{-2}, a_3, a_{-3} \}$ such that the corresponding Miyamoto involutions induce the following permutation actions (written in cycle notation) on the set $A$
\begin{align*}
    \tau(a_1) = \tau(a_{-1}): (a_2, \,  a_{-2})(a_3, \, a_{-3}); \\
    \tau(a_2) = \tau(a_{-2}): (a_1, \,  a_{-1})(a_3, \, a_{-3}); \\
    \tau(a_3) = \tau(a_{-3}): (a_1, \,  a_{-1})(a_2, \, a_{-2}).
\end{align*}
\end{defn}

\subsection{The shape of the algebra} 
\label{sec:shape}

Suppose that $V$ is a Frobenius $(\mathcal{M}, *)$-axial algebra that obeys property $\mathcal{P}$. We start by classifying the possible dihedral algebras contained in $V$. For each dihedral subalgebra, there is possibly more than one choice of generators for each dihedral algebra. For example, 
\[
\langle \langle a_1, a_2 \rangle \rangle = \langle \langle a_1, a_{-2} \rangle \rangle = \langle \langle a_{-1}, a_2 \rangle \rangle = \langle \langle a_{-1}, a_{-2} \rangle \rangle. 
\]

Moreover, from Lemma \ref{lem:inclusions}, the type of one dihedral algebra will determine the types of those algebras which it contains, and of those algebras that it is contained in, if any exist. For example, from Lemma \ref{lem:dihedral}, $U_0 := \langle \langle a_1, a_2 \rangle \rangle$ must be of type $4A$ or $4B$. This algebra contains the dihedral algebra $\langle \langle a_{-1}, a_{1} \rangle \rangle$ which must therefore be of type $2B$ or $2A$ if $U_0$ is of type $4A$ or $4B$ respectively. 

All non-trivial dihedral subalgebras of $V$ and their respective inclusions are shown in Figure 1. In this directed graph, the vertices are the non-trivial dihedral subalgebras of $V$ and if $U$ and $W$ are two such algebras then $U \rightarrow W$ if and only if $U$ is a subalgebra of $W$. 

It is clear from this graph that the choice of any one of these algebras determines the type of all other algebras. Thus we have only two cases - one when the algebra $\langle \langle a_1, a_2 \rangle \rangle$ is of type $4A$ and another when it is of type $4B$. For $X \in \{A, B \}$ we say that $V$ has \emph{shape} $4X$ if the subalgebra $U_0$ of $V$ is of type $4X$. 

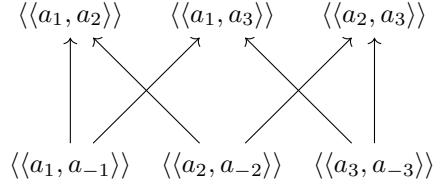
\begin{figure}
\begin{center}
\begin{tikzpicture}
\draw 
(0,2) node(a1) {$\langle \langle a_1, a_2 \rangle \rangle$}
(2,2) node(a2) {$\langle \langle a_1, a_3 \rangle \rangle$}
(4,2) node(a3) {$\langle \langle a_2, a_3 \rangle \rangle$}
(0,0) node(a4) {$\langle \langle a_1, a_{-1} \rangle \rangle$} 
(2,0) node(a5) {$\langle \langle a_2, a_{-2} \rangle \rangle$}
(4,0) node(a6) {$\langle \langle a_3, a_{-3} \rangle \rangle$};
\draw [->] (a4) -- (a1);
\draw [->] (a5) -- (a1);
\draw [->] (a4) -- (a2);
\draw [->] (a6) -- (a2);
\draw [->] (a5) -- (a3);
\draw [->] (a6) -- (a3);
\end{tikzpicture}
\end{center}
\label{fig:inclusions}
\caption{The inclusions of dihedral algebras in $M(U_{\mathcal{P}}, V_{\mathcal{P}})$}
\end{figure}

\subsection{The universal algebra}

In \cite{HRS15}, Hall, Rehren and Shpectorov define the category of commutative (non-associative) Frobenius algebras on $n$ (marked) generators, which they denote $\mathcal{C}_0$. They show that $\mathcal{C}_0$ contains an algebra $\hat{M}$, called the \emph{universal $n$-generated commutative Frobenius algebra}, such that $\mathcal{C}_0$ coincides with the category of all quotients of $\hat{M}$. 

The universal algebra $\hat{M}$ is constructed as follows. Let $\hat{X}$ be the free commutative non-associative magma with marked generators $\hat{A} := \{\hat{a}_1, \ldots, \hat{a}_n\}$. Let $\sim$ be the equivalence relation on $\hat{X} \times \hat{X}$ generated by all elementary equivalences $(\hat{x} \cdot \hat{y}, \hat{z}) \sim (\hat{x}, \hat{y} \cdot \hat{z})$ for $\hat{x}, \hat{y}, \hat{z} \in \hat{X}$. For $\hat{x}, \hat{y} \in \hat{X}$, we let $[\hat{x}, \hat{y}]$ denote the $\sim$-equivalence class containing $(\hat{x}, \hat{y})$. Furthermore, let $[\hat{X} \times \hat{X}]$ denote the set of all $\sim$-equivalence classes. 

Let 
\[
\hat{R} = \mathbf{k}\left[\{\lambda_{[\hat{x}, \hat{y}]} \}_{ [\hat{x}, \hat{y}] \in [\hat{X} \times \hat{X}]} \right]
\]
and let $\hat{M} = \hat{R}\hat{X}$, the set of all formal linear combinations $\sum_{\hat{x} \in \hat{X}} \alpha_{\hat{x}}\hat{x}$ where $\alpha_{\hat{x}} \in \hat{R}$. Then $\hat{M}$ is a commutative $\hat{R}$-algebra where the algebra product is defined using the operation in $\hat{X}$ and the distributive law. We can also define a bilinear form on $\hat{M}$ as follows:
\[
\left\langle \sum_{\hat{x} \in \hat{X}} \alpha_{\hat{x}}\hat{x}, \sum_{\hat{y} \in \hat{X}} \beta_{\hat{y}}\hat{y} \right\rangle = \sum_{\hat{x}, \hat{y} \in \hat{X}} \alpha_{\hat{x}}\beta_{\hat{y}} \lambda_{[\hat{x}, \hat{y}]}.
\]
Then, as $\lambda_{[\hat{x}, \hat{y}]}$ depends on the equivalence class $[\hat{x}, \hat{y}]$, this bilinear form is automatically Frobenius. This algebra $\hat{M}$ is also generated as an algebra by the marked generators $\hat{A} := \{\hat{a}_1, \ldots, \hat{a}_n\}$ of $\hat{X}$, 

\begin{prop}[{\cite[Proposition 4.1]{HRS15}}] 
Every algebra $M$ in $\mathcal{C}_0$ is the quotient, up to isomorphism, of $\hat{M}$ over an ideal $I_M$ and has coefficient ring $R$, which is the quotient of $\hat{R}$ over an ideal $J_M$.
\end{prop}

Choose $U \leq \hat{M}$ and $V \leq \hat{R}$. Then we let $\mathcal{C}_{U,V}$ be the subcategory of $\mathcal{C}_0$ consisting of all algebras $M$ from $\mathcal{C}_0$ such that $U \subseteq I_M$ and $V \subseteq J_M$. Let 
\begin{itemize}
\item $I_0$ be the ideal of $\hat{M}$ generated by $U$;
\item $J(U,V)$ be the ideal of $\hat{R}$ generated by $V$, together with all elements $\langle i, m \rangle$ for $i \in I_0$ and $m \in \hat{M}$;
\item $I(U,V)$ be the ideal of $\hat{M}$ generated by $U$ and $J(U,V)\hat{M}$.
\end{itemize}

\begin{prop}[{\cite[Proposition 4.4]{HRS15}}]
If $J(U,V) \neq \hat{R}$ then $\mathcal{C}_{U,V}$ is non-empty and contains a universal algebra $M(U,V)$ such that $\mathcal{C}_{U,V}$ consists of all quotients of $M(U,V)$. Moreover, $M(U,V) = \hat{M}/I(U,V)$ and has coefficient ring $R(U,V) = \hat{R}/J(U,V)$. 
\end{prop}

For a fixed fusion rule $(\mathcal{F}, *)$, Hall et al. show that the universal primitive Frobenius $(\mathcal{F}, *)$-axial algebra is equal to $M(U_{(\mathcal{F}, *)},V_{(\mathcal{F}, *)})$ for some explicit vlaues of $U_{(\mathcal{F}, *)}$ and $V_{(\mathcal{F}, *)}$. We will similarly construct the universal primitive Frobenius $(\mathcal{M}, *)$-axial algebra that obeys property $\mathcal{P}$ and has shape $4X$ for $X \in \{A,B\}$ (as defined in Section \ref{sec:shape}).

% The precise values of $U_{(\mathcal{F}, *)}$ and $V_{(\mathcal{F}, *)}$ are given in \cite[Theorem 5.1]{HRS15}. For example, as the $\hat{a}_i$ must be idempotents, we must have $\hat{a}_i \cdot \hat{a}_i - \hat{a}_i \in I_M$  for $1 \leq i \leq n$. The remaining axioms of Frobenius axial algebras can also be expressed in terms of elements of $U_{(\mathcal{F}, *)}$ and $V_{(\mathcal{F}, *)}$.

We now restrict to the case where  
\begin{itemize}
\item $\hat{M}$ has six marked generators, which we denote $ \hat{A} = \{\hat{a}_1, \hat{a}_{-1}, \hat{a}_2, \hat{a}_{-2}, \hat{a}_3, \hat{a}_{-3} \}$;
\item $\hat{R}$ contains $\mathbb{R}$ as a subfield.
\end{itemize}

% Recall that a Frobenius $(\mathcal{M}, *)$-axial algebra that obeys property $\mathcal{P}$ can have shape $4A$ or $4B$, as defined in Section \ref{sec:shape}.

\begin{lem}[{\cite[Lemma 5.4]{HRS15}}]
\label{lem:espaces}
Suppose that $V$ is an algebra over a field $\mathbf{k}$. If $u, v \in \hat{M}$ and $\Lambda \subseteq \mathbf{k}$ then 
\[
v \in \bigoplus_{\lambda \in \Lambda} V_\lambda^{(u)} \Leftrightarrow f_\Lambda(ad_u)(v) = 0
\]
where $V_\lambda^{(u)}$ denotes the $\lambda$-eigenspace of $ad_u$ and $f_\Lambda(x) = \prod_{\lambda \in \Lambda} (x - \lambda) \in \mathbf{k}[x]$.
\end{lem}

\begin{lem}
\label{lem:universal}
For $X \in \{A, B \}$, the primitive Frobenius $(\mathcal{M}, *)$-axial algebras that obey property $\mathcal{P}$ and that have shape $4X$ form a subcategory $\mathcal{C}_{4X}$ of $\mathcal{C}_0$. There exists a universal algebra $M_{4X}$ of $\mathcal{C}_{4X}$, and $\mathcal{C}_{4X}$ consists of all quotients of $M_{4X}$.  
\end{lem}

%In particular, we will show that if $M$ is in $\mathcal{C}_0$ then $M$ is a Frobenius $(\mathcal{M}, *)$-axial algebras that obeys property $\mathcal{P}$ and that has shape $4X$ if and only if $M$ is equal to $\hat{M}/I$ and has the coefficient ring $\hat{R}/J$ for $I$ and $J$ ideals such that $U \subseteq I$ and $V \subseteq J$. 

\begin{proof}
We will first define subsets $U \subseteq \hat{M}$ and $V \subseteq \hat{R}$ such that $\mathcal{C}_{4X} = \mathcal{C}(U,V)$ by expressing the conditions of property $\mathcal{P}$ as elements of $\hat{M}$ and $\hat{R}$. 

We let 
\begin{align*}
V_{\mathcal{P}} = V_{(\mathcal{M}, *)} \textrm{ and } U_{\mathcal{P}} = U_{(\mathcal{M}, *)} \cup \left\{ f_{\{\frac{1}{32}\}}(ad_{\hat{a}_{\pm i}})(\hat{a}_j - \hat{a}_{-j}) \right\} \cup \left\{ f_{\{1, 0, \frac{1}{4}\}}(ad_{\hat{a}_{\pm i}})(\hat{a}_j + \hat{a}_{-j}) \right\}
\end{align*}
where $i, j \in \{1, 2, 3\}$ and $i \neq j$. We will now show that if $M$ is an algebra in $\mathcal{C}(U_{(\mathcal{M}, *)}, V_{(\mathcal{M}, *)})$ then $M$ obeys property $\mathcal{P}$ if and only if $M$ also lies in $\mathcal{C}(U_{\mathcal{P}}, V_{\mathcal{P}})$.  

Fix $i, j \in \{1, 2, 3\}$ such that $i \neq j$ and suppose that $f_{\{\frac{1}{32}\}}(ad_{\hat{a}_{\pm i}})(\hat{a}_j - \hat{a}_{-j}) = 0$ and $f_{\{1, 0, \frac{1}{4}\}}(ad_{\hat{a}_{\pm i}})(\hat{a}_j + \hat{a}_{-j}) = 0$. Then, from Lemma \ref{lem:espaces},
\[
\hat{a}_j - \hat{a}_{-j} \in V_{\frac{1}{32}}^{(\hat{a}_{\pm i})} \textrm{ and } \hat{a}_j + \hat{a}_{-j} \in \bigoplus_{\lambda \in \{1, 0, \frac{1}{4}\}} V_\lambda^{(\hat{a}_{\pm i})}.
\]
and so $(\hat{a}_j - \hat{a}_{-j})^{\tau(a_{\pm i})} = \hat{a}_{-j} - \hat{a}_j$ and $(\hat{a}_j + \hat{a}_{-j})^{\tau(a_{\pm i})} = \hat{a}_j + \hat{a}_{-j}$. We can then infer that $\hat{a}_j^{\tau(a_{\pm i})} = \hat{a}_{-j}$. 

Conversely, if $\hat{a}_j^{\tau(a_{\pm i})} = \hat{a}_{-j}$ then it follows from the definition of $\tau(a_{\pm i})$ that $
\hat{a}_j - \hat{a}_{-j} \in V_{\frac{1}{32}}^{(\hat{a}_{\pm i})}$ and $\hat{a}_j + \hat{a}_{-j} \in \bigoplus_{\lambda \in \{1, 0, \frac{1}{4}\}} V_\lambda^{(\hat{a}_{\pm i})}$. Thus, from Lemma \ref{lem:espaces},
\[
f_{\{\frac{1}{32}\}}(ad_{\hat{a}_{\pm i}})(\hat{a}_j - \hat{a}_{-j}) = f_{\{1, 0, \frac{1}{4}\}}(ad_{\hat{a}_{\pm i}})(\hat{a}_j + \hat{a}_{-j}) = 0
\]
as required.

Futhermore, we note that requiring that an algebra has a certain shape is equivalent to putting restrictions on the elements of $\hat{M}$ and $\hat{R}$. These restrictions can be encoded as relations to be included in the sets $U$ and $V$. We do not explicitly give these relations; they can be read directly from the values in Table \ref{tab:IPSS10}. 

In particular, if we take $U$ and $V$ to be the union of the relations coming from the choice of the shape with $U_{\mathcal{P}}$ and $V_{\mathcal{P}}$ respectively then the category $\mathcal{C}(U,V)$ coincides with $\mathcal{C}_{4X}$. The final claim follows from \cite[Proposition 4.4]{HRS15}.
\end{proof}

Recall that $M_{4X}$ is equal to $\hat{M}/I(U,V)$ and has coefficient ring $R_{4X} = \hat{R}/J(U,V)$ for $U$ and $V$ as in Lemma \ref{lem:universal}. We let $\psi_{4X}$ and $\phi_{4X}$ denote the corresponding projections $\hat{M} \rightarrow M_{4X}$ and $\hat{R} \rightarrow R_{4X}$ and let $A := \{ a_1, a_{-1}, a_2, a_{-2}, a_3, a_{-3}\}$ denote the image of $\hat{A}$ under $\psi_{4X}$.

Note that the maps $\psi_{4X}$ and $\phi_{4X}$ determine the algebra product and Frobenius form on $M_{4X} = \psi_{4X}(\hat{M})$ in that
\begin{equation}
\label{eq:products}
u^{\psi_{4X}} \cdot v^{\psi_{4X}} =( u \cdot v)^{\psi_{4X}} \textrm{ and } \langle u^{\psi_{4X}}, v^{\psi_{4X}} \rangle = \langle u, v \rangle^{\phi_{4X}}
\end{equation}
for all $u, v \in \hat{M}$.

\subsection{Automorphisms of $M_{4X}$} 

Note that the group $G := \langle \tau(a) \mid a\in A \rangle \cong 2^2$ is a subgroup of $\mathrm{GL}(M_{4X})$ whose action on $M_{4X}$ preserves the algebra product and Frobenius form (from Propositions \ref{prop:algproduct} and \ref{prop:form}). This algebra also admits further symmetries, as we now explain. 

\begin{lem}
The permutation $\sigma \in \mathrm{Sym}(A)$ defined by
\[
\sigma = (a_1, \, a_2, \, a_3)(a_{-1}, \, a_{-2}, \, a_{-3})
\]
uniquely extends to an element $\psi_{\sigma} \in \mathrm{GL}(M_{4X})$ that preserves the algebra product and Frobenius form on $M_{4X}$ for $X \in \{A, B\}$. That is to say, there exists an automorphism $\phi_{\sigma}$ of $R_{4X}$ such that $\phi_{\sigma}$ preserves the subfield $\mathbb{R}$ of $R_{4X}$ and such that
\[
u^{\psi_{\sigma}} \cdot v^{\psi_{\sigma}} = (u \cdot v)^{\psi_{\sigma}} \textrm{ and } \langle u^{\psi_{\sigma}}, v^{\psi_{\sigma}} \rangle = \langle u, v \rangle^{\phi_{\sigma}} 
\]
\end{lem}

\begin{proof}
We first consider the algebra $\hat{M}$ and define the permutation $\hat{\sigma} \in \mathrm{Sym}(\hat{A})$ to be such that 
\[
\hat{\sigma} = (\hat{a}_1, \, \hat{a}_2, \, \hat{a}_3)(\hat{a}_{-1}, \, \hat{a}_{-2}, \, \hat{a}_{-3}).
\]
Then $\hat{\sigma}$ clearly extends to an automorphism $\hat{\psi}_\sigma$ of the free commutative magma $\hat{X}$. Moreover, it preserves the equivalence relation $\sim$ on $\hat{X} \times \hat{X}$ and so induces a permutation on the indeterminates $\lambda_{[\hat{x}, \hat{y}]}$ of the polynomial ring $\hat{R}$. This permutation induces an automorphism $\hat{\phi}_{\sigma}$ of $\hat{R}$ that fixes the field $\mathbb{R}$. We can now see that $\hat{\psi}_\sigma$ uniquely extends to $\hat{M}$ via
\[
\left( \sum_{\hat{x} \in \hat{X}} \alpha_{\hat{x}} \hat{x} \right)^{\hat{\psi}_{\sigma}} = \sum_{\hat{x} \in \hat{X}} \alpha_{\hat{x}}^{\hat{\phi}_{\sigma}} \hat{x}^{\hat{\psi}_{\sigma}}.
\]

Crucially, the maps $\hat{\psi}_\sigma$ and $\hat{\phi}_\sigma$ preserve the sets $U$ and $V$ from Lemma \ref{lem:universal} and therefore also the ideals $I(U,V)$ and $J(U,V)$. In particular, these maps descend to automorphisms of $M_{4X}$ and $R_{4X}$ as required.
\end{proof}

\section{The algebra $M_{4B}$}
\label{sec:4B}

\begin{thm}
The algebra $M_{4B}$ is $7$-dimensional. It has basis
\[
B := \{ a_1, a_{-1}, a_2, a_{-2}, a_3, a_{-3}, a_{\rho}\}
\]
where $a_{\rho} := a_1 + a_{-1} - 8a_1 \cdot a_{-1}$ and its coefficient ring $R_{4B}$ is equal to $\mathbb{R}$.
\end{thm}

\begin{proof}
The algebra $M_{4B}$ contains three distinct dihedral subalgebras of type $2A$; $\langle \langle a_i, a_{-i} \rangle \rangle$ for $1 \leq i \leq 3$. Each of these dihedral subalgebras is of dimension $3$ and contains a third basis vector $a_{\rho_i} := a_i + a_{-i} - 8 a_i \cdot a_{-i}$ for $i \in \{1, 2, 3\}$. 

However, if $i, j \in \{1, 2, 3\}$ such that $i \neq j$, then both $\langle \langle a_i, a_{-i} \rangle \rangle$ and $\langle \langle a_j, a_{-j} \rangle \rangle$ are contained in the dihedral algebra $\langle \langle a_i, a_j \rangle \rangle$, which is of type $4A$. In particular, this then implies that $a_{\rho_i} = a_{\rho_j}$. Thus we let $a_{\rho}$ denote the vector $a_{\rho_1} = a_{\rho_2} = a_{\rho_2}$.

Moreover, for all $a \in A$, there exists a dihedral algebra $U$ such that $a, a_{\rho} \in U$ and so the values for all algebra products on $B$ are given by the known values of the dihedral algebras. This the implies that the vector space is closed under multiplication and $M_{4B} = \langle B \rangle$. The value of the Frobenius form on $M_{4B}$ is also uniquely determined by the known values of the dihedral algebras and therefore the coefficient ring of the algebra is $\mathbb{R}$.

We can then use these values to check that $M_{4B}$ satisfies the definition of an axial algebra, and we are done.
\end{proof}

\section{The algebra $M_{4A}$}
\label{sec:4A}

\begin{thm}
\label{thm:4A}
The algebra $M_{4A}$ is $12$-dimensional. It has basis 
\[
B := A \cup \{ v_{(1,2)}, v_{(1,3)}, v_{(2,3)} \} \cup \{ a_1 \cdot v_{(2,3)}, a_2 \cdot v_{(1,3)},  a_3 \cdot v_{(1,2)}\}
\]
where 
\[
v_{(i,j)} = a_i + a_j + \frac{1}{3}a_{-i} + \frac{1}{3}a_{-j} - \frac{2^6}{3}a_i \cdot a_j
\]
for $i,j \in \{1, 2, 3\}$ such that $i \neq j$. Its coefficient ring $R_{4A}$ is equal to $\mathbb{R}[t]$ where $t := \langle a_1, v_{(2,3)} \rangle$. 
\end{thm}

\begin{thm}
\label{thm:4Ab}
There exists an infinite family of $12$-dimensional axial algebras of Monster type which we denote $\{M(t) \}_{t \in \mathbb{R}}$. If $t \notin \{0, \frac{1}{6}, \frac{9}{4} \}$ then the algebra $M(t)$ is simple. Otherwise, if $t \in \{0, \frac{1}{6}\}$, $M(t)$ contains a $3$-dimensional ideal and if $t = \frac{9}{4}$ then $M(t)$ contains a $5$-dimensional ideal.
\end{thm}

Henceforth, we write $M := M_{4A}$. 
 
\subsection{ Algebra products }

To begin, we let 
\[
\bar{B} := B \cup \{ a_{-1} \cdot v_{(2,3)}, a_{-2} \cdot v_{(1,3)},  a_{-3} \cdot v_{(1,2)}\}.
\]
Recall that the algebra product $\cdot$ and the Frobenius form $\langle \, , \, \rangle$ on $M$ are defined as in (\ref{eq:products}). Throughout this section, we will let $t := \langle a_1, v_{(2,3)} \rangle$.

\begin{prop}
\label{prop:frobeniusform1}
Some values of the Frobenius form $\langle \, , \, \rangle $ on the vectors of $\bar{B}$ are as given in Table 3.
\end{prop}

\begin{proof}
These values follow from the known values of the algebra products and Frobenius form on dihedral Majorana algebras and from the fact that for all $u, v, w \in M$, we must have $\langle u, v \cdot w \rangle = \langle u \cdot v, w \rangle$. 
\end{proof}

\begin{table}%
\begin{center}
\def\arraystretch{1.75}
\begin{tabular}{| >{$} l <{$} >{$} l <{$}  >{$} c <{$} | } \hline
u & v & \langle u,v \rangle \\ \hline
a_1 & v_{(2,3)} & t \\
v_{(1,2)} & v_{(2,3)} & -\frac{8t}{3} + \frac{1}{2} \\
a_1 & a_1 \cdot v_{(2,3)} & t \\
a_{-1} & a_1 \cdot v_{(2,3)} & 0 \\
a_2 & a_1 \cdot v_{(2,3)} & \frac{3t}{2^4} \\
v_{(1,2)} & a_1 \cdot v_{(2,3)} & -\frac{t}{2^2} \\
v_{(2,3)} & a_1 \cdot v_{(2,3)} & t \\ \hline
\end{tabular}
\end{center}
\label{tab:frobeniusform1}
\caption{Some values of the Frobenius form $\langle \, , \, \rangle$ on $M$}
\end{table}

\begin{prop}
\label{prop:products1}
\begin{align*}
v_{(1,2)} \cdot v_{(1,3)} = &-\frac{8t}{3}a_1 + \frac{1}{2^2}(v_{(1,2)} + v_{(1,3)} - v_{(2,3)}) + \frac{8}{3}a_1 \cdot v_{(2,3)}  \\
&- \frac{2}{3} ((a_2 + a_{-2}) \cdot v_{(1,3)} + (a_3 + a_{-3}) \cdot v_{(1,2)})  \\
a_1 \cdot (a_2 \cdot v_{(1,3)}) =& \frac{t}{2^3}a_1 + \frac{1}{2^4}a_1 \cdot v_{(2,3)} + \frac{1}{2^6}(5a_2 + 3a_{-2}) \cdot v_{(1,3)} - \frac{1}{2^4}(a_3 + a_{-3}) \cdot v_{(1,2)} \\ 
\end{align*}
\end{prop}

\begin{proof}
From the known values of dihedral algebras, the following are eigenvectors of $a_1$ 
\begin{align*}
\alpha_0 &:= -\frac{1}{2}a_1 + 2a_2 + 2a_{-2} + v_{(1,2)}  \in M_0^{(a_1)} \\
\beta_0 &:= -\frac{1}{3}(a_1 + a_{-1} + 2a_2 + 2a_{-2}) + v_{(1,2)} \in M_{\frac{1}{2^2}}^{(a_1)} \\
\alpha_1 &:= -\frac{1}{2}a_1 + 2a_3 + 2a_{-3} + v_{(1,3)} \in M_0^{(a_1)} \\
\beta_1 &:= -\frac{1}{3}(a_1 + a_{-1} + 2a_{-2} + 2a_{-3}) + v_{(1,3)} \in M_{\frac{1}{2^2}}^{(a_1)}.
\end{align*}

Then 
\[
(\alpha_0 - \beta_0) \cdot (\alpha_1 - \beta_1) = \frac{1}{2^2} a_{-1} + \frac{11}{2^2 \cdot 3}(a_2 + a_{-2} + a_3) - \frac{1}{2^3 \cdot 3}(v_{(1,2)} + v_{(1,3)}) - \frac{4}{3} v_{(2,3)} 
\]
and so the value of $a_1 \cdot ( (\alpha_0 - \beta_0) \cdot (\alpha_1 - \beta_1) )$ can also be computed. Moreover, from the fusion rule,
\[
a_1 \cdot ( (\alpha_0 - \beta_0) \cdot (\alpha_1 - \beta_1) ) = - \frac{1}{2^2}( \alpha_0 \cdot \beta_1 + \alpha_1 \cdot \beta_0) + \frac{1}{2^2} \langle \beta_0, \beta_1 \rangle a_1. 
\]

We calculate that $\langle \beta_0, \beta_1 \rangle = -\frac{16}{3}t + \frac{1}{3}$ and that 
\begin{align*}
\alpha_0 \cdot \beta_0 + \alpha_1 \cdot \beta_1 = 2 v_{(1,2)} \cdot v_{(1,3)} &- \frac{1}{2^2} \sum_{i = 1}^6 a_i - \frac{1}{2^3}(v_{(1,2)} + v_{(1,3)} - 4v_{(2,3)}) \\
&- \frac{4}{3}((a_2 + a_{-2}) \cdot v_{(1,3)} + (a_3 + a_{-3}) \cdot v_{(1,2)}). 
\end{align*}
Using these values, we can calculate $v_{(1,2)} \cdot v_{(1,3)}$ as required. 

The fusion rule again implies that
\[
a_1 \cdot ((\alpha_0 - \beta_0) \cdot \alpha_1) = -\frac{1}{2^2} \alpha_1 \cdot \beta_0
\]
As we now know the value of $v_{(1,2)} \cdot v_{(1,3)}$, the product $\alpha_1 \cdot \beta_0$ can be directly computed. 

Moreover,
\[
(\alpha_0 - \beta_0) \cdot \alpha_1 = -\frac{1}{2^3}(a_1 - a_{-1}) + \frac{7}{2^2 \cdot 3}(a_2 + a_{-2}) + \frac{2}{3}(a_3 + a_{-3}) + \frac{1}{2^3}v_{(1,2)} - v_{(2,3)} + \frac{8}{3}(a_2 + a_{-2}) \cdot v_{(1,3)}
\]
and so 
\begin{align*}
a_1 \cdot ((\alpha_0 - \beta_0) \cdot \alpha_1) = \frac{8}{3}a_1 \cdot ((a_2 &+ a_{-2}) \cdot v_{(1,3)}) + \frac{1}{2^5}(a_1 + a_{-1}) + \frac{1}{2^4 \cdot 3}(a_2 + a_{-2})\\
& + \frac{1}{2^3 \cdot 3}(a_3 + a_{-3}) - \frac{1}{2^5}(v_{(1,2)} + 2v_{(1,3)}) + a_1 \cdot v_{(2,3)}. 
\end{align*}
Using these values, we can calculate $a_1 \cdot ((a_2 + a_{-2}) \cdot v_{(1,3)})$. Finally, we know that $(a_2 - a_{-2}) \cdot v_{(1,3)} \in M_{\frac{1}{2^5}}^{(a_1)}$ and so 
\[
a_1 \cdot ((a_2 - a_{-2}) \cdot v_{(1,3)}) = \frac{1}{2^5} (a_2 - a_{-2}) \cdot v_{(1,3)}.
\]
We can then calculate the value of $a_1 \cdot (a_2 \cdot v_{(1,3)})$ as required.
\end{proof}

\begin{prop}
\label{prop:products2}
\begin{align*}
a_1 \cdot (a_1 \cdot v_{(2,3)}) =& \frac{3t}{2^2}a_1 + \frac{1}{2^2}a_1 \cdot v_{(2,3)} \\
a_{-1} \cdot (a_1 \cdot v_{(2,3)}) =& -\frac{t}{2^2}a_{-1} -\frac{t}{2^3} (a_3 - a_{-3}) + \frac{1}{2^2}a_{-1} \cdot v_{(2,3)} + \frac{1}{2^3}(a_3 + a_{-3}) \cdot v_{(1,3)} \\
v_{(1,2)} \cdot (a_1 \cdot v_{(2,3)}) =& -\frac{t}{2^3 \cdot 3}(7a_1 + a_{-1} + 9a_2 + 7a_{-2} + 2a_3 - 2a_{-3}) + \frac{t}{2^3}v_{(1,2)} \\
&+ \frac{1}{2^4}(7a_1 - 3a_{-1}) \cdot v_{(2,3)} +\frac{1}{2^2 \cdot 3} ((2a_2 + a_{-2}) \cdot v_{(1,3)} + (a_3 - a_{-3}) \cdot v_{(1,2)}).
\end{align*}
\end{prop}

\begin{proof}
If we take $\alpha_0, \alpha_1, \beta_0, \beta_1$ as in Proposition \ref{prop:products1} then, from the fusion rule, 
\begin{align*}
\alpha_{-1} &:= \frac{3}{2^4}(\alpha_0 \cdot \alpha_1 + \beta_0 \cdot \beta_1 - \frac{1}{2^2}(\beta_0, \beta_0)a_1 + \frac{1}{3}(\alpha_0 + \alpha_1)) \in M_0^{(a_1)} \\
\beta_{-1} &:= \frac{3}{2^4}(\alpha_0 \cdot \beta_1 + \alpha_1 \cdot \beta_0 - \frac{3}{2^3}(\beta_0 + \beta_0)) \in M_{\frac{1}{2^2}}^{(a_1)}.
\end{align*}
Using the value of $v_{(1,2)} \cdot v_{(1,3)}$ calculated in Proposition \ref{prop:products1}, we can calculate that
\begin{align*}
\alpha_{-1} &= -\frac{3t}{2^2}a_1 - \frac{1}{2^2} v_{(2,3)} + a_1 \cdot v_{(2,3)} \\
\beta_{-1} &= -ta_1 + a_1 \cdot v_{(2,3)}.
\end{align*}
The equality $a_1 \cdot \alpha_{-1} = 0$ then gives the value of $a_1 \cdot (a_1 \cdot v_{(2,3)})$ as required.

We further note that $\alpha_2 := a_{-1} \in M_0^{(a_1)}$ and so, from the fusion rule
\[
a_1 \cdot ((\alpha_{-1} - \beta_{-1}) \cdot (\alpha_0 + \beta_0 + \frac{5}{3} \alpha_2)) = \frac{1}{2^2}(\alpha_{-1} \cdot \beta_0 - (\alpha_0 + \frac{5}{3}\alpha_2) \cdot \beta_{-1}) + \frac{1}{2^2} \langle \beta_0, \beta_{-1} \rangle .
\]
We calculate that $\langle \beta_0, \beta_{-1} \rangle = -\frac{5t}{2 \cdot 3}$. Moreover, 
\begin{align*}
(\alpha_{-1} - \beta_{-1}) \cdot (\alpha_0 + \beta_0 + \frac{5}{3} \alpha_2) =  &-\frac{t}{2^4 \cdot 3}(a_1 + a_{-1}) + \frac{31t + 2}{2^3 \cdot 3} a_2 - \frac{t + 2}{2^2 \cdot 3} a_{-2} + \frac{1}{2^2 \cdot 3}(a_3 + a_{-3})  \\
&+ \frac{t - 2}{2^4}v_{(1,2)} + \frac{1}{2^3} v_{(1,3)} - \frac{1}{2^2}v_{(2,3)} + \frac{13}{2^3 \cdot 3} a_1 \cdot v_{(2,3)} \\
&+ \frac{4}{3} a_2 \cdot v_{(1,3)} + \frac{1}{3} (a_3 + a_{-3}) \cdot v_{(1,2)} .
\end{align*}

Then using the value of the algebra product $a_1 \cdot (a_1 \cdot v_{(2,3)})$, as well as those calculated in Proposition \ref{prop:products1}, we can also explicitly calculate the value of $a_1 \cdot ((\alpha_{-1} - \beta_{-1}) \cdot (\alpha_0 + \beta_0 + \frac{5}{3} \alpha_2))$. Moreover,
\begin{align*}
\alpha_{-1} \cdot \beta_0 - (\alpha_0 + \frac{5}{3}\alpha_2) \cdot \beta_{-1} = &-2a_{-1} \cdot  (a_1 \cdot v_{(2,3)}) + \frac{t}{2^4}(3a_1 + a_{-1}) + \frac{11t+ 1}{2^3 \cdot 3} a_2 - \frac{5t - 1}{2^3 \cdot 3} a_{-2}  \\
& - \frac{1}{2^3 \cdot 3}(a_3 + a_{-3}) - \frac{3t + 1}{2^4} v_{(1,2)} + \frac{1}{2^4} v_{(2,3)} - \frac{1}{2^3} a_1 \cdot v_{(2,3)}\\
& + \frac{1}{2} (a_2 + a_{-2}) \cdot v_{(1,3)} + \frac{1}{2 \cdot 3}(5a_3 + a_{-3}) \cdot v_{(1,2)}. 
\end{align*} 
Thus we can calculate the value of $a_{-1} \cdot  (a_1 \cdot v_{(2,3)})$ as required.

We now calculate that
\begin{align*}
(\alpha_{-1} - \beta_{-1}) \cdot a_0 = &\left(\frac{2t}{3} - \frac{1}{2^3}\right) a_2 - \frac{1}{2^3}(a_{-2} - a_3 - a_{-3}) - \frac{1}{2^4}(v_{(1,2)} - v_{(1,3)} + 4v_{(2,3)})\\
& + \frac{1}{2^3 \cdot 3}(7a_1 + 4a_{-1}) \cdot v_{(2,3)} - \frac{2}{3}a_2 \cdot v_{(1,3)} + \frac{1}{2 \cdot 3}(a_3 + a_{-3}) \cdot v_{(1,2)}.
\end{align*}
In particular, the value of $a_1 \cdot  (a_{-1} \cdot v_{(2,3)})$ is now known and so the value $a_1 \cdot ((\alpha_{-1} - \beta_{-1}) \cdot a_0)$ can be explicitly calculated. 

From the fusion rule, 
\[
a_1 \cdot ((\alpha_{-1} - \beta_{-1}) \cdot a_0) = -\frac{1}{2^2} \beta_{-1} \cdot \alpha_0
\]
and 
\begin{align*}
\beta_{-1} \cdot \alpha_0 = v_{(1,2)} \cdot (a_1 \cdot v_{(2,3)}) &-\frac{3t}{2^3} a_1 + \frac{t}{2^2}(a_2 + a_{-2}) + \frac{3}{2^4}(a_1 + a_{-1}) \cdot v_{(2,3)} \\ 
&- \frac{1}{2^3}(a_2 + a_{-2}) \cdot v_{(1,3)} + \frac{1}{2^2}(a_3 + a_{-3}) \cdot v_{(1,2)}.
\end{align*}
We can then use these values to calculate $v_{(1,2)} \cdot (a_1 \cdot v_{(2,3)})$ as required.
\end{proof}

\begin{prop}
\label{prop:nullspace}
The elements of $\bar{B}$ satisfy the following linear dependencies:
\begin{align*}
(a_1 - a_{-1}) \cdot v_{(2,3)} - t(a_1 - a_{-1}) = 0 \\
(a_2 - a_{-2}) \cdot v_{(1,3)} - t(a_2 - a_{-2}) = 0 \\
(a_3 - a_{-3}) \cdot v_{(1,2)} - t(a_3 - a_{-3}) = 0.
\end{align*}
\end{prop}

\begin{proof}
Recall that 
\begin{align*}
\alpha_{-1} &= -\frac{3t}{2^2}a_1 - \frac{1}{2^2} v_{(2,3)} + a_1 \cdot v_{(2,3)} \in M_0^{(a_1)} \\
\alpha_2 &= a_{-1} \in M_0^{(a_1)}.
\end{align*}
Using the value of $a_{-1} \cdot (a_1 \cdot v_{(2,3)})$ from Proposition \ref{prop:products2} and the fusion rule, we have
\[
\alpha_{-1} \cdot \alpha_2 = -\frac{t}{2^3}(2a_{-1} + a_3 + a_{-3}) + \frac{1}{2^3}(a_3 - a_{-3}) \cdot v_{(1,2)} \in M_0^{(a_1)}. 
\]
However, 
\[
a_1 \cdot (\alpha_{-1} \cdot \alpha_2) = -\frac{t}{2^8}(a_3 - a_{-3}) + \frac{1}{2^8}(a_3 - a_{-3}) \cdot v_{(1,2)}
\]
and so we must have $(a_3 - a_{-3}) \cdot v_{(1,2)} - t(a_3 - a_{-3}) = 0$. The symmetry between the three pairs of axes $\{a_1, a_{-1}\}$, $\{a_2, a_{-2}\}$ and $\{a_3, a_{-3}\}$ gives the remaining two relations.
\end{proof}

\begin{prop}
\label{prop:products3}
\begin{align*}
v_{(1,2)} \cdot (a_3 \cdot v_{(1,2)}) = \frac{(t - 1)t}{2^2}(a_3 - a_{-3}) + \frac{t}{2^2}v_{(1,2)} + \frac{1}{2} a_3 \cdot v_{(1,2)}
\end{align*}
\end{prop}

\begin{proof}

We now let 
\begin{align*}
\alpha_{-2} &:= \alpha_0 \cdot \alpha_1 - \frac{1}{2^2}(\alpha_0 + \alpha_1) - \frac{8}{3} \alpha_{-1} \\
&= -\frac{2t}{3} a_1 - \frac{1}{3} v_{(2,3)} + \frac{4}{3}(a_2 + a_{-2}) \cdot v_{(1,3)}+ \frac{4}{3}(a_3 + a_{-3}) \cdot v_{(1,2)} \in M_0^{(a_1)}. 
\end{align*}
We further note that
\[
a_1 \cdot ((\alpha_0 - \beta_0) \cdot \alpha_{-2}) = -\frac{1}{2^2} \beta_0 \cdot \alpha_{-2}.
\]
We can calculate that
\begin{align*}
(\alpha_0 - \beta_0) \cdot \alpha_{-2} = &-\frac{1}{2 \cdot 3}(a_1 - a_{-1}) + \frac{23t - 2}{3^2}(a_2 + a_{-2}) + \frac{2}{3^2}(a_3 + a_{-3}) + \frac{t}{2 \cdot 3}v_{(1,2)} -\frac{1}{3}v_{(2,3)} \\
&-\frac{2^3}{3^2}(a_1 + a_{-1}) \cdot v_{(2,3)} + \frac{20}{3^2}(a_2 + a_{-2}) \cdot v_{(1,3)} + \frac{2^3}{3^2}(a_3 + a_{-3}) \cdot v_{(1,2)}
\end{align*}
and so the value of $a_1 \cdot ((\alpha_0 - \beta_0) \cdot \alpha_{-2})$ can also be calculated.

Moreover, 
\begin{align*}
\beta_0 \cdot \alpha_{-2} = \frac{4}{3}v_{(1,2)} \cdot ((a_3 + a_{-3}) \cdot v_{(1,2)}) &-\frac{7t}{2 \cdot 3} a_1 - \frac{17t}{2 \cdot 3^2} a_{-1} - \frac{2t - 1}{2 \cdot 3^2} a_2 - \frac{18t - 1}{2 \cdot 3^2} a_{-2} \\
&+ \frac{1}{2 \cdot 3^2}(a_3 + a_{-3}) + \frac{2t -1}{2^2 \cdot 3}v_{(1,2)} +\frac{1}{2^2 \cdot 3} v_{(1,3)}  \\
&+ \frac{2}{3^2} (4a_1 + 3a_{-1}) \cdot v_{(2,3)}  - \frac{2}{3^2} (5a_2 + a_{-2}) \cdot v_{(1,3)}. 
\end{align*}
Thus we can explicitly calculate the value of $v_{(1,2)} \cdot ((a_3 + a_{-3}) \cdot v_{(1,2)})$. As $(a_3 - a_{-3}) \cdot v_{(1,2)} - t(a_3 - a_{-3}) = 0$, 
\[
v_{(1,2)} \cdot ((a_3 - a_{-3}) \cdot v_{(1,2)}) = t(a_3 - a_{-3}) \cdot v_{(1,2)}.
\]
We can use this value, along with that of $v_{(1,2)} \cdot ((a_3 + a_{-3}) \cdot v_{(1,2)})$, to calculate the product $v_{(1,2)} \cdot (a_3 \cdot v_{(1,2)})$ as required.

\end{proof}

\begin{prop}
\label{prop:products4}
\begin{align*}
(a_1 \cdot v_{(2,3)}) \cdot (a_1 \cdot v_{(2,3)}) &= \frac{(10t + 1)t}{2^4}a_1 - \frac{(2t - 1)t}{2^4}a_{-1} + \frac{t}{2^4}v_{(2,3)} + \frac{t}{2^2}a_1 \cdot v_{(2,3)} \\
(a_{-1} \cdot v_{(2,3)}) \cdot (a_1 \cdot v_{(2,3)}) &= -\frac{(6t - 1)t}{2^4}a_1 - \frac{(2t - 1)t}{2^4}a_{-1} + \frac{t}{2^4}v_{(2,3)} + \frac{t}{2^2}a_1 \cdot v_{(2,3)}
\end{align*}
\end{prop}

\begin{proof}
Recall that 
\begin{align*}
\alpha_{-1} &= -\frac{3t}{2^2}a_1 - \frac{1}{2^2} v_{(2,3)} + a_1 \cdot v_{(2,3)} \in M_0^{(a_1)} \\
\beta_{-1} &= -ta_1 + a_1 \cdot v_{(2,3)} \in M_{\frac{1}{2^2}}^{(a_1)}
\end{align*}
and so 
\[
a_1 \cdot ((\alpha_{-1} - \beta_{-1}) \cdot \alpha_{-1}) = -\frac{1}{2^2} \beta_{-1} \cdot \alpha_{-1}. 
\]
As 
\[
(\alpha_{-1} - \beta_{-1}) \cdot \alpha_{-1} = -\frac{(2t - 1)t}{2^4}(a_1 - a_{-1}) -\frac{t - 1}{2^4}a_1 \cdot v_{(2,3)} + \frac{3t - 2}{2^4} a_{-1} \cdot v_{(2,3)},
\]
we can explicitly calculate the value of $a_1 \cdot ((\alpha_{-1} - \beta_{-1}) \cdot \alpha_{-1})$.

Moreover, using the value of $v_{(2,3)} \cdot (a_1 \cdot v_{(2,3)})$ from Proposition \ref{prop:products3}, we can calculate that
\[
\beta_{-1} \cdot \alpha_{-1} = (a_1 \cdot v_{(2,3)}) \cdot (a_1 \cdot v_{(2,3)}) - \frac{(11t -1)t}{2^4} a_1 - \frac{(2t -1)t}{2^4}a_{-1} + \frac{t}{2^4}v_{(2,3)} + \frac{3t + 2}{2^4} a_1 \cdot v_{(2,3)}.
\]
We can then use these values to find the value of $(a_1 \cdot v_{(2,3)}) \cdot (a_1 \cdot v_{(2,3)})$ as required.

As we have the linear dependency $(a_1 - a_{-1}) \cdot v_{(2,3)} - t(a_1 - a_{-1}) = 0$, we must have
\[
(a_{-1} \cdot v_{(2,3)}) \cdot (a_1 \cdot v_{(2,3)}) = (a_1 \cdot v_{(2,3)}) \cdot (a_1 \cdot v_{(2,3)}) - t(a_1 - a_{-1}) \cdot (a_1 \cdot v_{(2,3)}).
\]
All these product values are known and so we have the value of $(a_{-1} \cdot v_{(2,3)}) \cdot (a_1 \cdot v_{(2,3)})$ as required.
\end{proof}

\begin{prop}
\label{prop:products5}
\begin{align*}
(a_2 \cdot v_{(1,3)}) \cdot (a_1 \cdot v_{(2,3)}) = \frac{t^2}{2^5}(a_1 - a_{-1} + a_2 - a_{-2}) + \frac{(2t - 1)t}{2^5}a_3 - \frac{(2t + 1)t}{2^5}a_{-3} + \frac{t}{2^5}v_{(1,2)} &\\
+ \frac{t}{2^6}(v_{(1,3)} + v_{(2,3)}) + \frac{t}{2^3}(a_1 \cdot v_{(1,2)} + a_2 \cdot v_{(1,3)} - a_3 \cdot v_{(1,2)})&.
\end{align*}
\end{prop}

\begin{proof}
Recall that
\[
\alpha_{-2} := -\frac{2t}{3} a_1 - \frac{1}{3} v_{(2,3)} + \frac{4}{3}(a_2 + a_{-2}) \cdot v_{(1,3)}+ \frac{4}{3}(a_3 + a_{-3}) \cdot v_{(1,2)} \in M_0^{(a_1)}.
\]
We also let
\begin{align*}
\beta_{-2} &:= \alpha_0 \cdot \beta_1 - \frac{1}{2^2}\beta_0 - \frac{1}{2^3}\beta_1 - \frac{8}{3}\beta_{-1} = \frac{4}{3}((a_2 + a_{-2}) \cdot v_{(1,3)} - (a_3, a_{-3}) \cdot v_{(1,2)}) \in M_{\frac{1}{2^2}}^{(a_1)}.
\end{align*}
Then
\[
a_1 \cdot ((\alpha_{-1} - \beta_{-1}) \cdot (\alpha_{-2} -\beta_{-2})) = -\frac{1}{2^2}(\alpha_{-1} \cdot \beta_{-2} + \alpha_{-2} \cdot \beta_{-1}) - \frac{1}{2^2} \langle \beta_{-1}, \beta_{-2} \rangle a_1.
\]
Both $\alpha_{-2}$ and $\beta_{-2}$ can be rewritten using the linear dependencies given in Proposition \ref{prop:nullspace} to give
\begin{align*}
\alpha_{-2} &:= -\frac{2t}{3} a_1 - \frac{4t}{3}(a_2 - a_{-2} + a_3 - a_{-3}) -\frac{1}{3}v_{(2,3)} + \frac{8}{3}(a_2 \cdot v_{(1,3)} + a_3 \cdot v_{(1,2)}) \\
\beta_{-2} &:= - \frac{4t}{3}(a_2 - a_{-2} + a_3 - a_{-3}) + \frac{8}{3}(a_2 \cdot v_{(1,3)} - a_3 \cdot v_{(1,2)}).
\end{align*}

Using these new representations, we calculate that 
\begin{align*}
(\alpha_{-1} - \beta_{-1}) \cdot (\alpha_{-2} -\beta_{-2}) &= \frac{(3t + 22)t}{2^2 \cdot 3^2}a_2 - \frac{(3t - 10)t}{2^2 \cdot 3^2}a_{-2} - \frac{(3t - 14)t}{2^2 \cdot 3^2} a_3 + \frac{(3t + 2)t}{2^2 \cdot 3^2}a_{-3} \\
&- \frac{2t - 1}{2^2 \cdot 3}v_{(2,3)} + \frac{t}{2 \cdot 3} a_1 \cdot v_{(2,3)} + \frac{t}{2}(a_2 \cdot v_{(1,3)} - a_3 \cdot v_{(1,2)}).
\end{align*}
and so the value of $a_1 \cdot ((\alpha_{-1} - \beta_{-1}) \cdot (\alpha_{-2} -\beta_{-2}))$ may also be explicitly calculated.

Next, we calculate that $\langle \beta_{-1}, \beta_{-2} \rangle = 0$ and 
\begin{align*}
\alpha_{-1} \cdot \beta_{-2} + \alpha_{-2} \cdot \beta_{-1}  &= \frac{16}{3}(a_1 \cdot v_{(2,3)}) \cdot (a_2 \cdot v_{(1,3)}) - \frac{(8t - 1)t}{2^2 \cdot 3} a_1 + \frac{(2t - 1)t}{2^2 \cdot 3}a_{-1} \\
&- \frac{(9t + 4)t}{2^2 \cdot 3^2}(a_2 - a_{-3}) + \frac{(9t - 4)t}{2^2 \cdot 3^2}(a_{-2} - a_3) - \frac{t}{2^2 \cdot 3} v_{(2,3)} \\
&- \frac{t + 1}{2 \cdot 3} (a_2 \cdot v_{(1,3)} - a_3 \cdot v_{(1,2)}).
\end{align*}
Thus we can use these values to calculate the product $(a_1 \cdot v_{(2,3)}) \cdot (a_2 \cdot v_{(1,3)})$ as required.
\end{proof}

At this stage, we have all algebra products of the form $u \cdot v$ for $u,v \in \bar{B}$. We can now calculate the remaining values of the Frobenius form on $M$. 

\begin{prop}
\label{prop:frobeniusform2}
Some values of the Frobenius form $\langle \, , \, \rangle$ on the vectors of $\bar{B}$ are as given in Table 4.
\end{prop}

\begin{proof}
These values follow from the fact that for all $u, v, w \in M$, we must have $\langle u, v \cdot w \rangle = \langle u \cdot v, w \rangle$ and from the algebra product values determined in Proposition \ref{prop:products2}. 
\end{proof}

\begin{table}%
\begin{center}
\def\arraystretch{1.75}
\begin{tabular}{| >{$} l <{$} >{$} l <{$}  >{$} c <{$} | } \hline
u & v & \langle u,v \rangle \\ \hline
a_1 \cdot v_{(2,3)} & a_1 \cdot v_{(2,3)} & \frac{(3t +1)t}{2^2} \\
a_1 \cdot v_{(2,3)} & a_2 \cdot v_{(1,3)} & \frac{(2t +1)t}{2^4} \\ \hline
\end{tabular}
\end{center}
\label{tab:frobeniusform2}
\caption{Some values of the Frobenius form $\langle \, , \, \rangle$ on $M$}
\end{table}

\subsection{The algebra structure}

\begin{proof}[Proof of Theorem \ref{thm:4A}]
We have shown in Propositions \ref{prop:products1}, \ref{prop:products2}, \ref{prop:products3}, \ref{prop:products4} and \ref{prop:products5} that the algebra product on $M$ is closed on the set
\[
\bar{B} := B \cup \{ a_{-1} \cdot v_{(2,3)}, a_{-2} \cdot v_{(1,3)},  a_{-3} \cdot v_{(1,2)}\}.
\]
Moreover, from Proposition \ref{prop:nullspace}, the elements of $\bar{B}$ satisfy the following linear dependencies:
\begin{align*}
(a_1 - a_{-1}) \cdot v_{(2,3)} - t(a_1 - a_{-1}) = 0 \\
(a_2 - a_{-2}) \cdot v_{(1,3)} - t(a_2 - a_{-2}) = 0 \\
(a_3 - a_{-3}) \cdot v_{(1,2)} - t(a_3 - a_{-3}) = 0.
\end{align*}
Thus we can take the following to be a basis of $M$:
\[
B = \{a_1, a_{-1}, a_2, a_{-2}, a_3, a_{-3},  v_{(1,2)}, v_{(1,3)}, v_{(2,3)},  a_1 \cdot v_{(2,3)},  a_2 \cdot v_{(1,3)}, a_3 \cdot v_{(1,2)} \}.
\]

The algebra product values of certain pairs of elements of $B$ are given in Table 5. With the exception of products given by the values of dihedral algebras, from the eight products given in this table, all other products of pairs of basis vectors can be recovered from the action of $G = \langle \tau(a) \mid a\in A \rangle$ and $\sigma$. 
 
From these products, we can calculate the eigenspaces of the adjoint action of the axis $a_1$. With the exception of the one-dimensional $1$-eigenspace, these are given in Table 6. Again, the eigenvectors for the remaining $5$ axes can be recovered from the action of $G$ and $\sigma$. We can check using these eigenvectors and the products in Table 5 that $M$ obeys the Monster fusion rule. 

Finally, we use these algebra products to check that the values in Tables 3 and 4 do indeed give a Frobenius form on $M$. 
\end{proof}

Theorem \ref{thm:4A} shows that there exists a $12$-dimensional axial algebra $M_{4A}$ of Monster type over the ring $\mathbb{R}[t]$. If we let $t$ take any value in $\mathbb{R}$ then this gives a $12$-dimensional axial algebra of Monster type over the field $\mathbb{R}$ that we denote $M(t)$. 

In order to prove Theorem \ref{thm:4Ab}, we first require a few additional definitions and results.

\begin{defn}
Let $(V,A)$ be a primitive axial algebra of Monster type that admits a Frobenius form $\langle \, , \, \rangle$. Then we define the \emph{projection graph} of $V$ to be the graph with vertices $A$ such that $a_0, a_1 \in A$ are joined by an edge if and only if $\langle a_0, a_1 \rangle \neq 0$  
\end{defn}

The projection graph of $M(t)$ is given in Figure 2. 

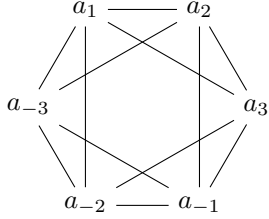
\begin{figure}
\begin{center}
\begin{tikzpicture}
\foreach \i in {1,2,3}{
    \draw (180 - \i*60: 1.5cm) node (a\i) {$a_{\i}$};
    }
\foreach \i in {1,2,3}{
    \draw (360 - \i*60: 1.5cm) node (am\i) {$a_{-\i}$};
    }
\foreach \i in {2,3}{
    \draw (a1) -- (a\i);
    \draw (a1) -- (am\i);
    \draw (am1) -- (a\i);
    \draw (am1) -- (am\i);
    }
\draw (a2) -- (a3);
\draw (a2) -- (am3);
\draw (am2) -- (a3);
\draw (am2) -- (am3);
\end{tikzpicture}
\end{center}
\label{fig:projection}
\caption{The projection graph of $M(t)$}
\end{figure}

\begin{defn}
Let $(V,A)$ be a primitive axial algebra. The \emph{radical} $R(V,A)$ of $V$ with respect to $A$ is the unique largest ideal of $V$ containing no axes from $A$. 
\end{defn} 

\begin{thm}[{\cite[Theorem 4.11]{KMS18}}]
\label{thm:radical}
Let $(V,A)$ be a primitive axial algebra that admits a Frobenius form $\langle \, , \, \rangle$. Then $R(V,A) = V^{\perp}$ where $V^{\perp} = \{ u \in V \mid \langle u, v \rangle = 0 \textrm{ for all } v \in V\}$.
\end{thm}

\begin{lem}
\label{lem:ideal}
The radical $M(t)^{\perp}$ of the Frobenius form on $M(t)$ is the unique maximal ideal of $M(t)$ for all $t \in \mathbb{R}$.
\end{lem}
\begin{proof}
Let $I$ be a proper ideal of $M(t)$. We will show that $I$ must be contained in $M(t)^{\perp}$. Suppose for contradiction that there exist an axis $a$ such that $a \in I$. As $I$ is proper, there must exist at least one further axis $b$ such that $b \notin I$. Then, from Lemma \ref{lem:projection}, we can write
\[
a = \langle a, b \rangle b + a_0 + a_{\frac{1}{4}} + a_{\frac{1}{32}} 
\]
where $a_\mu \in V_\mu^{(b)}$. 

Let $\Lambda := \{0, \frac{1}{4}, \frac{1}{32}\}$ and consider $f_\Lambda(x) \in \mathbb{R}[x]$ as defined in Lemma \ref{lem:espaces}. As $a \in I$ and $I$ is an ideal, $f_\Lambda(ad_b)(a) \in I$. Moreover, from Lemma \ref{lem:espaces}, 
\[
f_\Lambda(ad_b)(a) = \langle a, b \rangle f_\Lambda(ad_b)(b) + f_\Lambda(ad_b)(a_0 + a_{\frac{1}{4}} + a_{\frac{1}{32}} ) =  f_\Lambda(1)\langle a, b \rangle b
\]
and $f_\Lambda(1) \neq 0$. Thus, as $b \notin I$, we must have $\langle a, b \rangle = 0$. 

As this is true for all $a, b \in A$ such that $a \in I$ and all $b \notin I$, the projection graph of $M(t)$ must be disconnected. We can see from Figure 2 that this is not the case. Thus we can conclude that $I$ contains no axes from $A$ and so from Theorem \ref{thm:radical} is contained in $M(t)^\perp$ as required.
\end{proof}

\begin{proof}[Proof of Theorem \ref{thm:4Ab}]
From Lemma \ref{lem:ideal}, the algebra $M(t)$ is simple exactly when the Frobenius form $\langle \, , \, \rangle$ is non-degenerate, or equivalently when the determinant of the Gram matrix of $\langle \, , \, \rangle$ on $B$ is non-zero.

Using the values of the Frobenius form in Tables 3 and 4, we have calculated that the determinant of this Gram matrix is equal to 
\[
-\frac{t^3 (6 t - 1)^3 (4 t - 9)^6 }{2^{19} \cdot 3^3}.
\]
Thus if $t \notin \{0, \frac{1}{6}, \frac{9}{4} \}$ then the algebra $M(t)$ is simple. If $t = 0$ or $t = \frac{1}{6}$ then we calculate that $M(t)$ contains a $3$-dimensional ideal and if $t = \frac{9}{4}$ then $M(t)$ contains a $5$-dimensional ideal.
\end{proof}

\begin{landscape}

\begin{table}%
\begin{center}
\def\arraystretch{1.75}
\begin{tabular}{| >{$} l <{$} >{$} l <{$}  >{$} c <{$} | } \hline
u & v & u \cdot v \\ \hline
v_{(1,2)} &  v_{(1,3)} & \def\arraystretch{1.2}\begin{tabular}{ >{$} c <{$}} -\frac{8t}{3} a_1 + \frac{2t}{3}(a_2 - a_{-2} + a_3 - a_{-3}) + \frac{1}{2^2} (v_{(1,2)} + v_{(1,3)} - v_{(2,3)})\\ - \frac{4}{3}(2a_1 \cdot v_{(2,3)} - a_2 \cdot v_{(1,3)} - a_3 \cdot v_{(1,2)}) \end{tabular} \def\arraystretch{1.75}\\
a_1 & a_1 \cdot v_{(2,3)} & \frac{3t}{2^2}a_1 + \frac{1}{2^2} a_1 \cdot v_{(2,3)} \\
a_{-1} & a_1 \cdot v_{(2,3)} & -\frac{t}{2^2} a_1 + \frac{1}{2^2} a_1 \cdot v_{(2,3)} \\a
a_2 & a_1 \cdot v_{(2,3)} & -\frac{3t}{2^8}(a_1 - a_{-1}) + \frac{t}{2^4}(2a_2 + a_3 - a_{-3}) + \frac{1}{2^4}(2a_1 \cdot v_{(2,3)} + a_2 \cdot v_{(1,3)} - a_3 \cdot v_{(1,2)}) \\
v_{(1,2)} & a_1 \cdot v_{(2,3)} & -\frac{5t}{2^4 \cdot 3}a_1 - \frac{11t}{2^4 \cdot 3}a_{-1} - \frac{11t}{2^3 \cdot 3}a_2 - \frac{5t}{2^3 \cdot 3}a_{-2} + \frac{t}{2^3} v_{(1,2)} + \frac{1}{2^2}(a_1 \cdot v_{(2,3)} + a_2 \cdot v_{(1,3)}) \\
v_{(2,3)} & a_1 \cdot v_{(2,3)} & \frac{(2t-1)t}{2^2}(a_1 - a_{-1}) + \frac{1}{2^2}v_{(2,3)} + \frac{1}{2} a_1 \cdot v_{(2,3)} \\
a_1 \cdot v_{(2,3)} & a_1 \cdot v_{(2,3)} & \frac{(10t + 1)t}{2^4}a_1 - \frac{(2t - 1)t}{2^4} a_{-1} + \frac{t}{2^4}v_{(2,3)} + \frac{t}{2^2} a_1 \cdot v_{(2,3)} \\
a_1 \cdot v_{(2,3)} & a_2 \cdot v_{(1,3)} & \def\arraystretch{1.2} \begin{tabular}{>{$} c <{$}}\frac{t}{2^5}(a_1 - a_{-1} + a_2 - a_{-2}) + \frac{(2t - 1)t}{2^5} a_3 - \frac{(2t + 1)t}{2^5} a_{-3} + \frac{t}{2^6}(2v_{(1,2)} + v_{(1,3)} + v_{(2,3)}) \\+ \frac{t}{2^3}(a_1 \cdot v_{(2,3)} + a_2 \cdot v_{(1,3)} - a_3 \cdot v_{(1,2)}) \end{tabular} \\ \hline
\end{tabular}
\end{center}
\label{tab:algebraproducts}
\caption{Representative algebra products on $M$}
\end{table}

\begin{table}%
\begin{center}
\def\arraystretch{1.75}
\begin{tabular}{| >{$} c <{$} | >{$} c <{$} | >{$} c <{$} | } \hline

0 & \frac{1}{4} & \frac{1}{32} \\\hline 

\begin{tabular}{>{$} c <{$}}

\def\arraystretch{1}
\begin{tabular}{>{$} c <{$}}
-\frac{t}{2^2}a_1 - \frac{t}{2}(a_2 - a_{-2} + a_3 - a_{-3}) -\frac{1}{2^3}v_{(2,3)} \\
+ a_2 \cdot v_{(1,3)} + a_3 \cdot v_{(1,2)} \end{tabular}\\
\def\arraystretch{1.75}
-\frac{3t}{2^2}a_1 - \frac{1}{2^2} v_{(2,3)} + a_1 \cdot v_{(2,3)} \\
-\frac{1}{2} a_1 + 2a_2 + 2a_{-2} + v_{(1,2)} \\
-\frac{1}{2} a_1 + 2a_3 + 2a_{-3} + v_{(1,2)} \\
a_{-1} \\ 
\end{tabular}

&

\begin{tabular}{>{$} c <{$}}

\def\arraystretch{1}
\begin{tabular}{>{$} c <{$}}
\frac{t}{2}(a_2 - a_{-2} + a_3 - a_{-3}) - a_2 \cdot v_{(1,3)}\\
 + a_3 \cdot v_{(a_1, a_2)} \end{tabular}\\
\def\arraystretch{1.75}

-ta_1 + a_1 \cdot v_{(2,3)} \\
-\frac{1}{3}(a_1 + a_{-1}) -\frac{2}{3}(a_2 - a_{-2}) + v_{(1,2)} \\
-\frac{1}{3}(a_1 + a_{-1}) -\frac{2}{3}(a_3 - a_{-3}) + v_{(1,3)} \\ 
\end{tabular}

&

\begin{tabular}{>{$} c <{$}}
a_2 - a_{-2} \\
a_3 - a_{-3} \\ 
\end{tabular} \\ \hline
\end{tabular}
\end{center}
\label{tab:evecs}
\caption{Eigenvectors of the axis $a_1$ in $M$}
\end{table}

\end{landscape}

\section{The Frobenius form on the algebra $M_{4A}$}
\label{sec:frobenius}

In Section \ref{sec:4A}, we constructed an infinite family of axial algebras of Monster type over the field $\mathbb{R}$ that we denote $\{M(t)\}_{t \in \mathbb{R}}$. We now ask which of these algebras are also Majorana algebras. 

\begin{defn}
Let $(V, A)$ be a primitive axial algebra of Monster type over $\mathbb{R}$ that admits a Frobenius form $\langle \, , \, \rangle$. Then $(V, A)$ is a \emph{Majorana algebra} if both of the following hold
\begin{enumerate}[i)]
\item the Frobenius form $\langle \, , \, \rangle$ is positive definite;
\item $V$ obeys \emph{Norton's inequality}, i.e. $\langle u \cdot u, v \cdot v \rangle \geq \langle u \cdot v, u \cdot v \rangle$ for all $u, v \in V$.
\end{enumerate}
\end{defn}

The following is the main result of this section.

\begin{thm}
\label{thm:majorana}
The axial algebra $M(t)$ (as constructed in Section \ref{sec:4A}) is a Majorana algebra if and only if $t \in (0,\frac{1}{6})$. If $t \in \{0, \frac{1}{6}\}$ then there exists a $9$-dimensional quotient of $M(t)$ that is a Majorana algebra. 
\end{thm}

Norton's inequality is sometimes referred to as axiom M2, in reference to the Majorana axioms as introduced in \cite[Chapter 9]{Ivanov09}. Whilst it is part of the definition of a Majorana algebra, it is not used in the construction of any of the known Majorana algebras and it is an open problem as to whether an axial algebra of Monster type that admits a positive definite Frobenius form must necessarily obey Norton's inequality.

Theorem \ref{thm:majorana} shows that there exists axial algebras of Monster type that admit a Frobenius form but which do not obey Norton's inequality, i.e. the algebras $M(t)$ where $t \notin [0, \frac{1}{6}]$. These are the first examples of such algebras. Moreover, we show that an algebra of this form admits a Frobenius form that is positive definite precisely when $t \in (0, \frac{1}{6})$. This is suggests that an axial algebra of Monster type that admits a positive (semi)definite Frobenius form must obey axiom M2. 

In order to prove Theorem \ref{thm:majorana}, we will require the following preliminary results.

\begin{lem}[{\cite[Lemma 7.8]{ISe12}}]
\label{lem:M2}
Let $V$ be an $n$-dimensional algebra with commutative algebra product $\cdot$ and bilinear form $\langle \, , \,\rangle $. Let $\{v_i \, : \, 1 \leq i  \leq n\}$ be a basis of $V$, and define a $(n^2 \times n^2)$-dimensional matrix $\mathcal{B} = (b_{ij,kl})$ in the following way. The rows and columns are indexed by the ordered pairs $(i,j)$ for $1 \leq i,j \leq n$ and 
\[
b_{ij,kl} = \langle v_i \cdot v_k, v_j \cdot v_l\rangle  - \langle v_j \cdot v_k, v_i \cdot v_l\rangle .
\]
Then $V$ satisfies Norton's inequality if and only if $\mathcal{B}$ is positive semidefinite.
\end{lem}

\begin{proof}
For $u,v \in V$, write $u$ and $v$ as linear combinations
\[
u = \sum_{i=1}^n \lambda_i v_i \textrm{ and } v = \sum_{j=1}^n \mu_j v_j
\]
and form the $n^2$-long vector $z$ with entries $\lambda_i \mu_j$. In this vector, the coordinate $\lambda_i\mu_j$ is in the position indexed by $(i,j)$ in the matrix $\mathcal{B}$. Then the inequality $\langle u \cdot u, v \cdot v\rangle  - \langle u \cdot v, u \cdot v\rangle  \geq 0 $ is equivalent to $z \mathcal{B} z^T \geq 0$. Hence, if $\mathcal{B}$ is positive semidefinite then Norton's inequality must hold in $V$.
\end{proof}

\begin{thm}[{\cite{GV13}}]
\label{thm:pd}
A symmetric matrix $A$ is positive definite (respectively positive semidefinite) if and only if it can be decomposed as
\[
A := LDL^T
\]
where $L$ is a lower triangular matrix and $D$ is a diagonal matrix whose diagonal elements are all positive (respectively non-negative).
\end{thm}

The construction of the matrices $L$ and $D$ is known as \emph{LDLT decomposition} and we have implemented an algorithm in GAP that performs this decomposition for positive semidefinite matrices. This can be accessed as part of the package \texttt{MajoranaAlgebras} \cite{PW18b}.

\begin{prop}
\label{prop:innerproduct}
The Frobenius form $\langle \, , \, \rangle$ on the algebra $M(t)$ is positive semidefinite if and only if $t \in [0, \frac{1}{6}]$. The form is positive semidefinite but not positive definite if and only if  $t \in \{0, \frac{1}{6}\}$.
\end{prop}

\begin{proof}
We use the values of the Frobenius form calculated in Section \ref{sec:4A} to construct the Gram matrix $\mathcal{G}$ of the Frobenius form on $M(t)$ with respect to the spanning set 
\[
B = A \cup \{ v_{(1,2)}, v_{(1,3)}, v_{(2,3)} \} \cup \{ a_1 \cdot v_{(2,3)}, a_2 \cdot v_{(1,3)},  a_3 \cdot v_{(1,2)}\}.
\]
Then the form will be positive semidefinite if and only if $\mathcal{G}$ is positive semidefinite. 

We use LDLT decomposition to calculate a lower triangular matrix $L$ and a diagonal matrix $D$ such that $\mathcal{G} = LDL^T$. We have calculated that
\[
D = \mathrm{diag} \{ 1, 1, \sfrac{511}{512}, \sfrac{510}{511}, \sfrac{271}{272}, \sfrac{270}{271}, r_1, r_2, r_3, r_4, r_5, r_6 \}
\]
where $r_1, \dots, r_6$ are rational functions whose values are given in Table \ref{tab:polyinner}.

For each rational function $r_i$, we have used \cite{mathematica} to calculate the range in $\mathbb{R}$ on which $r_i$ takes non-negative values. These are given in the final column of Table \ref{tab:polyinner}. It is easy to check that the intersection of these ranges is $[0, \frac{1}{6}]$ and that the $r_i$ all take positive values on $(0, \frac{1}{6})$.
\end{proof}

% TODO change positive to non-negative

\begin{table}
\renewcommand\arraystretch{3}
\begin{center}
\vspace{0.35cm}
\noindent
\begin{tabular}{|>{$}c<{$} >{$}c<{$} >{$}c<{$}|} \hline
i & r_i & \textrm{Values of } t \textrm{ s.t. }  0 \leq r_i[t] < \infty \\ \hline
1 & - \frac{ 272 }{ 135 }  t^2 + \frac{ 8 }{ 45 }  t + \frac{ 22 }{ 15 } & \left[ \frac{3 - 15 \sqrt{15}}{68}, \frac{3 + 15 \sqrt{15}}{68} \right] \\
2 & \myfrac[4pt]{  t^4 + \frac{ 1 }{ 16 }  t^3 - \frac{ 717 }{ 128 }  t^2 + \frac{ 171 }{ 256 }  t + \frac{ 1053 }{ 2048 }  }{ - \frac{ 255 }{ 512 }  t^2 + \frac{ 45 }{ 1024 }  t + \frac{ 1485 }{ 4096 }  } & \left[ \frac{-39}{16}, \frac{3 - 15 \sqrt{15}}{68} \right), \left[-\frac{1}{4}, \frac{3}{8} \right], \left(\frac{3 + 15 \sqrt{15}}{68}, \frac{9}{4} \right] \\
3 & \myfrac[4pt]{  t^4 + \frac{ 5 }{ 2 }  t^3 - \frac{ 171 }{ 16 }  t^2 - \frac{ 9 }{ 32 }  t + \frac{ 81 }{ 128 }  }{ - \frac{ 1 }{ 2 }  t^2 - \frac{ 33 }{ 32 }  t + \frac{ 117 }{ 256 }  } & \left[ \frac{-3 \sqrt{11} - 9}{4}, -\frac{39}{16} \right), \left[-\frac{1}{4}, \frac{3 \sqrt{11} - 9}{4} \right], \left( \frac{3}{8}, \frac{9}{4} \right] \\
4 & \myfrac[4pt]{  t^5 - \frac{ 9 }{ 4 }  t^4 - \frac{ 7 }{ 9 }  t^3 + \frac{ 15 }{ 8 }  t^2 - \frac{ 9 }{ 32 }  t  }{ 4 t^3 + 19 t^2 - \frac{ 9 }{ 8 }   } & 
\renewcommand\arraystretch{1.5}\ml{c}{ \left(-\infty, \frac{-9 - 3\sqrt{11}}{4} \right), \left[ \frac{-1 - \sqrt{109}}{12}, -\frac{1}{4} \right), \left[0, \frac{1}{6} \right], \\ {\left( \frac{3 \sqrt{11} - 9}{4}, \frac{\sqrt{109} - 1}{12} \right]}, \left[ \frac{9}{4}, \infty \right) }\\
5 & \myfrac[4pt]{  t^5 - \frac{ 133 }{ 36 }  t^4 + \frac{ 613 }{ 216 }  t^3 + \frac{ 33 }{ 32 }  t^2 - \frac{ 15 }{ 64 }  t  }{ \frac{ 16 }{ 3 }  t^3 + \frac{ 20 }{ 9 }  t^2 - \frac{ 34 }{ 9 }  t - 1  } & \renewcommand\arraystretch{1.5}\ml{c}{ \left(-\infty, \frac{-1 - \sqrt{109}}{12} \right), \left[ \frac{23 - \sqrt{1339}}{36}, - \frac{1}{4} \right), \left[0, \frac{1}{6} \right],  \\ {\left( \frac{\sqrt{109} -  1}{12}, \frac{23 + \sqrt{1339}}{36} \right]}, \left[ \frac{9}{4}, \infty \right) }\\
6 & \myfrac[4pt]{  t^4 - \frac{ 14 }{ 3 }  t^3 + \frac{ 93 }{ 16 }  t^2 - \frac{ 27 }{ 32 }  t  }{ 6 t^2 - \frac{ 23 }{ 3 }  t - \frac{ 15 }{ 4 }  } & \left(\infty, \frac{23 - \sqrt{1339}}{36} \right), \left[0, \frac{1}{6} \right], \left( \frac{23 + \sqrt{1339}}{36}, \infty \right) \\ \hline
\end{tabular}
\end{center}
\label{tab:polyinner}
\caption{Diagonal entries for the Gram matrix of the Frobenius form on $M_{4A}$}
\end{table}

\begin{prop}
\label{prop:norton}
The algebra $M(t)$ obeys Norton's inequality if and only if if $t \in [0, \frac{1}{6}]$.
\end{prop}

\begin{proof}
We construct the matrix $\mathcal{B}$ as described in Lemma \ref{lem:M2} and calculate its LDLT decomposition $\mathcal{B} = LDL^T$. The algebra $M(t)$ obeys axiom M2 if and only if this matrix is positive semidefinite or, equivalently, if and only if all diagonal entries of the matrix $D$ are non-negative. We have calculated that the diagonal elements of $D$ consist of the rational numbers
\[
0, \sfrac{15}{632}, \sfrac{107}{4096}, \sfrac{395}{15872}, \sfrac{1395}{54784}
\]
as well as nineteen rational functions that we label $s_1, \dots, s_{19}$. We have calculated using \cite{mathematica} that $s_i(t) \geq 0$ for all $i \in [1, \dots,  19]$ if and only if $t \in [0, \frac{1}{6}]$. Here we do not explicitly give the rational functions in question. 
\end{proof}

\begin{proof}[Proof of Theorem \ref{thm:majorana}]
From Propositions \ref{prop:innerproduct} and \ref{prop:norton}, the algebra $M(t)$ is a Majorana algebra if and only if $t \in (0, \frac{1}{6})$. If $t \in \{0, \frac{1}{6} \}$ then the Frobenius form is positive semidefinite. We have calculated that the radical of this form is a $3$-dimensional ideal of $M(t)$. Taking the quotient of $M(t)$ by this ideal gives a $9$-dimensional primitive axial algebra of Monster type that admits a positive definite Frobenius form. We have also checked that Norton's inequality holds on this algebra, and so this algebra is indeed a Majorana algebra. 
\end{proof}

\section{The $4A$ axes in $M_{4A}$}
\label{sec:4Afusion}

Finally, we will study the $4A$ axes in the algebra $M_{4A}$. 

\begin{thm}
\label{thm:4Afusion}
For $i$ and $j$ such that $1 \leq i < j \leq 3$, vector $v_{(i,j)} \in M_{4A}$ is an idempotent whose eigenspace decomposition satisfies the fusion rule $(\mathcal{F}_{4A}^{(t)}, *)$ as given by the table below.

\begin{center}
\def\arraystretch{1.5}
\begin{tabular}{>{$}c<{$}|>{$}c<{$}>{$}c<{$}>{$}c<{$}>{$}c<{$}>{$}c<{$}}
            & 1 & 0 & \frac{1}{2} & \frac{3}{8}         & t \\ \hline
1           & 1 &  \emptyset & \frac{1}{2} & \frac{3}{8}         & t \\
0           & \emptyset  & 0 & \frac{1}{2} & \frac{3}{8}         & t \\
\frac{1}{2} & \frac{1}{2}  & \frac{1}{2}  & 1, 0   & \frac{3}{8} & t \\
\frac{3}{8} & \frac{3}{8}  & \frac{3}{8}  & \frac{3}{8} & 1, 0, \frac{1}{2}   & \emptyset  \\
t           & t   & t  & t   & \emptyset  & 1, 0, \frac{1}{2} \\ 
\end{tabular}
\end{center} 

Moreover, if $t \notin \{1, 0, \frac{1}{2}, \frac{3}{8} \}$ then this fusion rule admits a $C_2 \times C_2$-grading 
\[
\mathrm{gr}^{(t)}_{4A}: \langle a, b \mid a^2 = b^2 = (ab)^2 = 1 \rangle \rightarrow P(\mathcal{F}^{(t)}_{4A})
\]
such that
\begin{align*}
\mathrm{gr}^{(t)}_{4A}: 1 &\mapsto \left\{1, 0, \frac{1}{2}\right\} \\
a & \mapsto \left\{\frac{3}{8}\right\} \\
b & \mapsto \{ t \} \\
ab & \mapsto \emptyset.
\end{align*}
Finally, if $t = \frac{3}{8}$ then this fusion rule admits a $C_2$-grading
\[
\mathrm{gr}^{(\sfrac{3}{8})}_{4A}: \langle a \mid a^2 = 1 \rangle \rightarrow P(\mathcal{F}^{(\sfrac{3}{8})}_{4A})
\]
such that
\begin{align*}
\mathrm{gr}^{(\sfrac{3}{8})}_{4A}: 1 &\mapsto \left\{1, 0, \frac{1}{2}\right\} \\
a & \mapsto \left\{\frac{3}{8}\right\}.
\end{align*}
\end{thm}

\begin{proof}
Using the algebra product values calculated in Section \ref{sec:4A}, we calculate the eigenspace decomposition of the vector $v_{(i,j)}$ and check that this satisfies the fusion rule $(\mathcal{F}^{(t)}_{4A}, *)$ for all $t \in \mathbb{R}$ such that $t \notin \{1, 0, \frac{1}{2}\}$. This eigenspace decomposition is given in Table \ref{tab:4A}. 
\end{proof}

\begin{coro}
The subalgebra $ \langle \langle v_{(1,2)}, v_{(1,3)}, v_{(2,3)} \rangle \rangle$ of $M_{4A}$ is a $9$-dimensional $(\mathcal{F}_{4A}, *)$-axial algebra spanned by $ v_{(1,2)}$, $v_{(1,3)}$ and $v_{(2,3)}$ as well as the vectors
\[
\left\{ t a_i - a_i \cdot v_{(j,k)} \mid \{i, j, k\} = \{1, 2, 3\} \right\} \cup \left\{ t a_{-i} + a_i \cdot v_{(j,k)} \mid \{i, j, k\} = \{1, 2, 3\}\right\}.
\]
This algebra is an axial algebra of Jordan type $\frac{1}{2}$ i.e. the vectors $v_{(i,j)}$ satisfy the fusion rule $\mathcal{J}(\frac{1}{2})$, as given by the table below.

\begin{center}
\def\arraystretch{1.5}
\begin{tabular}{>{$}c<{$}|>{$}c<{$}>{$}c<{$}>{$}c<{$}}
            & 1 & 0 & \frac{1}{2}  \\ \hline
1           & 1 &  \emptyset & \frac{1}{2}  \\
0           & \emptyset  & 0 & \frac{1}{2}  \\
\frac{1}{2} & \frac{1}{2}  & \frac{1}{2}  & 1, 0  \\

\end{tabular}
\end{center} 
\end{coro}

\begin{proof}
We can check using the algebra product values calculated in Section \ref{sec:4A} that the smallest subalgebra of $M_{4A}$ containing the vectors $ v_{(1,2)}$, $v_{(1,3)}$ and $v_{(2,3)}$ is the $9$-dimensional algebra with basis given above. It is then straightforward to calculate that this algebra coincides with the direct sum of the $1$, $0$ and $\frac{1}{2}$-eigenspaces of the adjoint action of $v_{(i,j)}$ on $M_{4A}$ for all $i, j$ such that $1 \leq i < j \leq 3$. The fusion rule follows from the fusion rule already calculated in Theorem \ref{thm:4Afusion}. 
\end{proof}

\begin{landscape}

\begin{table}%
\begin{center}
\def\arraystretch{1.75}
\begin{tabular}{| >{$} c <{$} | >{$} c <{$} | >{$} c <{$} | >{$} c <{$} | } \hline

0 & \frac{1}{2} & \frac{3}{8} & t \\\hline 

\begin{tabular}{>{$} c <{$}}
-\frac{4}{3}( a_i + a_{-i} + a_j + a_{-j}) + v_{(i,j)} \\
(\frac{1}{8} -\frac{5t}{6})a_i + (\frac{1}{8} + \frac{t}{6})a_{-i} + \frac{1}{8}(a_j + a_{-j}) - \frac{1}{4}(a_k + a_{-k}) + a_i \cdot v_{(j,k)} \\
\frac{1}{8}(a_i + a_{-i}) + (\frac{1}{8} -\frac{5t}{6})a_j + (\frac{1}{8} + \frac{t}{6})a_{-j} - \frac{1}{4}(a_k + a_{-k}) + a_j \cdot v_{(i,k)} \\
-\frac{t}{3}( a_i + a_{-i} + a_j + a_{-j}) - \frac{2t + 1}{4}(a_k + a_{-k}) + a_k \cdot v_{(i,j)} 
\end{tabular}

& 

\begin{tabular}{>{$} c <{$}}
a_i + a_{-i} - a_j - a_{-j} \\
-\frac{3t}{2} a_i - \frac{t}{2}a_{-i} + \frac{t}{2}v_{(i,j)} + \frac{1}{16}(v_{(i,k)} - v_{(j,k)}) + a_i \cdot v_{(j,k)} \\
-\frac{3t}{2} a_j - \frac{t}{2}a_{-j} + \frac{t}{2}v_{(i,j)} - \frac{1}{16}(v_{(i,k)} - v_{(j,k)}) + a_j \cdot v_{(i,k)} \\
-\frac{t}{2}(a_k - a_{-k}) - \frac{t}{2}v_{(i,j)} + a_k \cdot v_{(i,j)} \\
\end{tabular}

& 

\begin{tabular}{>{$} c <{$}}
a_i - a_{-i} \\
a_j - a_{-j} \\
\end{tabular}

& 

\begin{tabular}{>{$} c <{$}}
a_k - a_{-k} \\
\end{tabular} \\

\hline 
\end{tabular}
\label{tab:4A}
\caption{Eigenvectors of the axis $v_{(i,j)}$ in $M_{4A}$}
\end{center}
\end{table}

\end{landscape}

\section*{Acknowledgements} I am very grateful to Prof. S.~Shpectorov for his helpful comments and advice in the writing of this paper. This project particularly benefited from conversations that took place at the Symmetry vs Regularity conference in Pilsen in July 2018 and for this I would also like to thank the organisers of this conference.

\end{document}